\documentclass[12pt, a4paper, reqno]{amsart}
\usepackage[margin=2cm]{geometry}
\newcommand{\todo}[1]{}

\usepackage[colorlinks=true,linkcolor=blue,citecolor=blue,urlcolor=blue,citebordercolor={0 0 1},urlbordercolor={0 0 1},linkbordercolor={0 0 1}]{hyperref} 

\makeatletter
\DeclareMathAlphabet{\mathpzc}{OT1}{pzc}{m}{it}

\usepackage{amsmath, amssymb, amsfonts, amstext, verbatim, amsthm, mathrsfs,mathabx,aliascnt,epsfig,enumerate,xspace}
\usepackage[mathcal]{eucal}
\usepackage[usenames]{color}
\usepackage[all,arc,knot,poly]{xy}

\theoremstyle{plain}

\newcommand{\refnewtheoremn}[4]{
\newaliascnt{#1}{#2}
\newtheorem{#1}[#1]{#3}
\aliascntresetthe{#1}
\expandafter\providecommand\csname #1autorefname\endcsname{#4}}

\newcommand{\refnewtheorem}[3]{\refnewtheoremn{#1}{#2}{#3}{#3}}

\def\makeCal#1{
\expandafter\newcommand\csname c#1\endcsname{\mathcal{#1}}}
\def\makeBB#1{
\expandafter\newcommand\csname b#1\endcsname{\mathbb{#1}}}
\def\makeFrak#1{
\expandafter\newcommand\csname f#1\endcsname{\mathfrak{#1}}}

\count@=0 \loop \advance\count@ 1 \edef\y{\@Alph\count@}
\expandafter\makeCal\y \expandafter\makeBB\y \expandafter\makeFrak\y
\ifnum\count@<26 \repeat

\newtheorem{thm}{Theorem}[section]
\refnewtheorem{lemma}{thm}{Lemma}
\refnewtheorem{assumptions}{thm}{Assumptions}
\refnewtheorem{examples}{thm}{Examples}
\refnewtheorem{claim}{thm}{Claim}
\refnewtheorem{cor}{thm}{Corollary}
\refnewtheorem{conj}{thm}{Conjecture}
\refnewtheorem{prop}{thm}{Proposition}
\refnewtheorem{proposition}{thm}{Proposition}
\refnewtheorem{quest}{thm}{Question}
\refnewtheorem{assumption}{thm}{Assumption}

\theoremstyle{definition}
\refnewtheorem{remark}{thm}{Remark}
\refnewtheorem{remarks}{thm}{Remarks}
\refnewtheorem{properties}{thm}{Properties}
\refnewtheorem{defn}{thm}{Definition}
\refnewtheorem{definition}{thm}{Definition}
\refnewtheorem{notn}{thm}{Notation}
\refnewtheorem{const}{thm}{Construction}
\refnewtheorem{example}{thm}{Example}
\refnewtheorem{fact}{thm}{Fact}
\refnewtheorem{problem}{thm}{Problem}

\renewcommand{\geq}{\geqslant}
\renewcommand{\leq}{\leqslant}
\newcommand{\into}{\hookrightarrow}

\newcommand {\id}{\operatorname{id}}

\newcommand{\Stab}{\operatorname{Stab}}
\renewcommand{\Im}{\operatorname{Im}}
\renewcommand{\Re}{\operatorname{Re}}
\newcommand{\Hom}{\operatorname{Hom}}

\newcommand{\Aut}{\operatorname{Aut}}
\newcommand{\Quad}{\operatorname{Quad}}
\newcommand{\Proj}{\operatorname{Proj}}
\newcommand {\<}{\langle}
\renewcommand {\>}{\rangle}

\newcommand{\tensor}{\otimes}

\newcommand{\isom}{\cong}
\newcommand{\hk}{hyperk{\"a}hler }

\renewcommand{\sf}{\operatorname{sf}}

\newcommand{\lra}{\longrightarrow}
\newcommand{\lRa}[1]{\stackrel{#1}{\longrightarrow}}
\newcommand{\Tw}{\operatorname{Tw}}
\newcommand{\Br}{\operatorname{Br}}
\newcommand{\Diag}{\operatorname{Diag}}
\newcommand{\Pol}{\cP}
\newcommand{\Poly}{\operatorname{Poly}}
\newcommand{\PGL}{\operatorname{PGL}}
\newcommand{\expl}{\exp_{\cL}}
\newcommand{\expt}{\exp_{\cT}}
\newcommand{\fr}{\operatorname{fr}}
\newcommand{\st}{\operatorname{star^0}}
\newcommand{\inte}{\operatorname{int}}
\renewcommand{\H}{\widebar{\cH}}
\newcommand{\Hn}{\widebar{\cH}^n}
\newcommand{\Hhalf}{\widetilde{\cH}}
\newcommand{\Hhalfn}{\widetilde{\cH}^n}

\title{Clusters, twistors and stability conditions I}
\author{Tom Bridgeland}
\author{Helge Ruddat}

\date{}

\begin{document}

\setcounter{tocdepth}{1}

\begin{abstract}
We consider a quiver $Q$ of ADE type and use cluster combinatorics  to define two complex manifolds $\cS$ and $\cL$. The space $\cS$ can be identified with a quotient of the space of stability conditions on the CY$_3$ category associated to $Q$. The space $\cL$ has a canonical map to the complex cluster Poisson space $\cX_\bC$ which we prove to be a local homeomorphism. When $Q$ is of type $A$,  we give a geometric description of the spaces $\cS$ and $\cL$ as moduli spaces of meromorphic quadratic differentials and projective structures respectively. In the sequel paper we will introduce a space $\pi\colon Z\to \bC$ whose fibre over over a point $\epsilon\in \bC$ is isomorphic to  $\cS$  when $\epsilon=0$ and to  $\cL$ otherwise. The problem of constructing sections of this map gives a  geometric approach  to the Donaldson--Thomas  Riemann--Hilbert problems of \cite{RHDT1}. 
\end{abstract}

\maketitle

\tableofcontents

\newpage

\section{Introduction}

Let $Q$ be a quiver whose underlying unoriented graph is a simply-laced Dynkin diagram. We denote by $\cX_{\bC}$ the complex analytic space associated to the  cluster Poisson variety \cite{FG1} of $Q$. We refer to it simply as the \emph{complex cluster space}, and view it as a possibly non-Hausdorff complex manifold. It is  obtained by  gluing overlapping copies of the algebraic tori $(\bC^*)^n$  using certain birational maps.
The combinatorics of this process is controlled by the cluster exchange graph via the standard mutation formulae  \cite{FG1,FZ1} which will be recalled below.

The aim of this paper is to use the same combinatorics to define two other complex manifolds: the \emph{cluster stability space}  $\cS$ and the \emph{log cluster space} $\cL$.
In the sequel \cite{sequel} we will define a third complex manifold $\cZ$, which we call the \emph{cluster twistor space}. It comes equipped with a map  $\pi\colon \cZ\to \bC$ whose fibre  over a point $\epsilon\in \bC$ is isomorphic to  $\cS$  when $\epsilon=0$ and to  $\cL$ otherwise. 

The cluster stability space $\cS$ gives a  purely combinatorial approach to the space of stability conditions \cite{Stab,Ludmil} of the CY$_3$ triangulated category associated to $Q$. On the other hand, the log cluster space $\cL$ comes equipped with a map $\expl\colon \cL\to \cX_{\bC}$ which we prove to be a local homeomorphism. Thus, to first approximation, our work exhibits the space of stability conditions as a kind of specialisation of the complex cluster  space.

When $Q$ is of type $A$ or $D$  we can give  descriptions of our spaces $\cS$ and $\cL$ in the same vein as Fock and Goncharov's description of $\cX_{\bC}$ as a space of framed local systems \cite{FG0}. Work of one of us with Smith \cite{BS} shows that the space $\cS$ can be identified with a space of framed meromorphic quadratic differentials on $\bP^1$. More recent work of Gupta and Mj \cite{GM} exhibits $\cL$ as a space of framed meromorphic projective structures. The map $\expl\colon \cL\to \cX_\bC$  then sends a framed projective structure to the framed local system defined by its monodromy and Stokes data.

One motivation for our work is a project which aims to encode the Donaldson-Thomas (DT) invariants of a CY$_3$ category in a complex \hk  structure on the total space of the tangent bundle of the space of stability conditions \cite{RHDT2,Strachan,BJoy}. Any complex \hk manifold has a twistor space $\pi\colon \cZ\to \bP^1$, and in the case of the CY$_3$ category associated to an ADE quiver we expect that our cluster twistor space is the restriction of this twistor space to  $\bC\subset \bP^1$. 

\subsection{Summary of the spaces}

Let us introduce the semi-closed and closed upper half-planes
\[\Hhalf=\{z\in \bC: 0\leq \arg(z)< \pi\}, \qquad \H= \{z\in \bC:\Im(z)\geq 0\}.\]
The  complex cluster  space $\cX_{\bC}$ is  constructed by  gluing overlapping copies of  $(\bC^*)^n$  using birational maps of the form
\begin{equation}\label{x}x_j\mapsto \begin{cases} 
x_j\, (1+x_i^{+1})^{v_{ij}} &\text{ if } i\neq j\text{ and }v_{ij}\geq  0, \\ x_j\, (1+x_i^{-1})^{v_{ij}} &\text{ if } i\neq j\text{ and }v_{ij}<0, \\
x_i^{-1} &\text{ if }i=j.\end{cases}\end{equation}
The cluster stability space $\cS$ is built by gluing copies of $\Hhalfn$  along their boundaries using  maps
\begin{equation}
\label{trop_glue_intro} w_j\mapsto 
\begin{cases} w_j +v_{ij} \cdot w_i & \text{ if }i\neq j\text{ and } v_{ij}\geq 0 \text{ and } \Re(w_i)\geq 0, \\ 
w_j -v_{ij} \cdot w_i & \text{ if }i\neq j\text{ and } v_{ij}\leq 0 \text{ and } \Re(w_i)\leq 0, \\
w_j  & \text{ if } i\neq j\text{ and } \Re(v_{ij} w_i)< 0, \\
-w_i &\text{ if }i=j.\end{cases}\end{equation}

The log cluster space $\cL$ is obtained by gluing copies of  $\Hn$ along theirboundaries using maps
\begin{equation}\label{aintro} w_j \mapsto 
\begin{cases} w_j+v_{ij}\cdot \log\big(1+e^{+w_i}\big) &\text{ if } i\neq j\text{ and }v_{ij}\geq 0, \\
w_j+v_{ij}\cdot\log \big(1+e^{-w_i}\big) &\text{ if } i\neq j\text{ and }v_{ij}<0, \\
-w_i &\text{ if }i=j.
\end{cases}\end{equation}
Setting $x_i=\exp(w_i)$ defines a continuous map
$\expl\colon \cL\to\cX_{\bC}$.

In the sequel to this paper we will construct the cluster twistor space $\cZ$   using maps
\begin{equation}\label{ab} w_j \mapsto 
\begin{cases} w_j+\epsilon\cdot v_{ij}\cdot \log\big(1+e^{+w_i/\epsilon}\big) &\text{ if } i\neq j\text{ and }v_{ij}\geq 0, \\
w_j+\epsilon\cdot v_{ij}\cdot\log \big(1+e^{-w_i/\epsilon}\big) &\text{ if } i\neq j\text{ and }v_{ij}<0, \\
-w_i &\text{ if }i=j.
\end{cases}\end{equation}
This  reproduces \eqref{aintro} when $\epsilon=1$ and limits to \eqref{trop_glue_intro} as $\epsilon\to 0$ along the positive real axis.

\subsection{Statement of results}

Recall that $Q$ is a quiver whose underlying unoriented graph is a simply-laced Dynkin diagram.  We denote by $\cX_{\bR_+}\subset \cX_{\bC}$ the positive real cluster Poisson variety,  and by $\cX_{\bR^t}$ the tropical cluster variety. Then $\cX_{\bR_+}$ is a topological manifold with a collection of homeomorphisms to $\bR_+^n$, and $\cX_{\bR^t}$ is a piecewise linear (PL)  manifold with a collection of PL homeomorphisms to $\bR^n$. There is a well-defined subset $\cX_{\bZ^t}\subset \cX_{\bR^t}$ of integral tropical points. 

Let $\cD$ be the bounded derived category of  the CY$_3$ Ginzburg dg-algebra $\Pi_3(Q)$.  It is a $\bC$-linear CY$_3$ triangulated category with a distinguished bounded t-structure whose heart $\cA\subset \cD$ is equivalent to the category of  representations of $Q$. We denote by $\Aut(\cD)$ the group of $\bC$-linear triangulated auto-equivalences of $\cD$, and by $\Br(\cD)$ the subgroup generated by spherical twists in the simple objects of $\cA$. The space of stability conditions $\Stab(\cD)$ is a complex manifold  which carries a natural action of $\Aut(\cD)$.

Each of the cluster spaces defined above carries a  natural action of a discrete group $\bG$ called the  cluster modular group. In the approach to cluster theory using CY$_3$ triangulated categories, this group is the quotient $\bG=\Aut(\cD)/\Br(\cD)$.  It therefore has a distinguished element $T\in \bG$, known as the DT element,  corresponding  to the shift functor $[1]\in \Aut(\cD)$.  

\begin{thm}
\label{clstab}
    The cluster stability space $\cS$ is a complex manifold with the following properties:
    \begin{itemize}
    \item[(i)] as a topological space, $\cS$ can be naturally identified with an open subset of $\cX_{\bR^t}\times \cX_{\bR^t}$,
    \item[(ii)] there is an integral affine structure on $\cS$ and an  action of the additive group $\bC$,
    \item[(iii)] the action of $\pi i\in \bC$ coincides with the action of the DT element $T\in \bG$,
    \item[(iv)] as a  complex manifold, $\cS$ is isomorphic to the quotient $\Stab(\cD)/\Br(\cD)$.
    \end{itemize}
\end{thm}

\begin{thm}
\label{thisone}
    The log cluster space $\cL$ is a complex manifold with the following properties:
    \begin{itemize}
    \item[(i)] as a topological space, $\cL$ can be naturally identified with $\cX_{\bR_+}\times \cX_{\bR^t}$,
        \item[(ii)]the map $\expl\colon \cL\to \cX_{\bC}$ is holomorphic and a local homeomorphism,
        \item[(iii)] the fibre of $\expl$ over a point of $\cX_{\bR_+}\subset \cX_{\bC}$ is naturally identified with the set $\cX_{\bZ^t}$,
        \item[(iv)] there is a $C^1$-homeomorphism $h\colon \cL\to T\,\cX_{\bR_+}$, compatible with the projections to $\cX_{\bR_+}$.
        \end{itemize}
\end{thm}

To prove Theorem \ref{thisone}  we apply the inverse function theorem to the composition $\exp_{\cL}\circ h^{-1}$ of the maps from (ii) and (iv). To prove that the derivative of this map is an isomorphism we first construct a complete fan in each tangent space $T_x \cX_{\bR_+}$. This is obtained by taking the inverse images of the positive orthants under the derivatives of the cluster charts. To show that these cones fit together to give a fan we continuously deform the gluing maps \eqref{x} in a one-parameter family  so that at $t=0$ they become monomial. We then show that this deformation induces what we call a  {combinatorially constant family of fans}. 

\subsection{Description in type A}

Consider the case when $Q$ is an orientation of the A$_n$ Dynkin diagram, and set $m=n+3$. Similar results can be obtained when the quiver $Q$ has type $D$. The analogue of the polynomial \eqref{polyi} is  a rational function  with a double pole at $x=0$ and no other finite poles.

Work of one of us with Smith \cite{BS} shows that the space $\cS$ can be identified with a space of framed meromorphic quadratic differentials on $\bP^1$ with a single pole at infinity. More explicitly, let $\Pol\isom \bC^n$ denote the space of polynomials of the form
\begin{equation}
    \label{polyi}
q(x)=x^{n+1}+a_{n-1}x^{n-1}+\cdots +a_1 x + a_0,\qquad a_i\in \bC,\end{equation}
and let $\Pol_0\subset \Pol$ be the open subset of polynomials without repeated zeroes. Associated to such a polynomial is a quadratic differential $q(x)\, dx^{\tensor 2}$.  Using the results of  \cite[Section 12.1]{BS} we have

\begin{thm}
\label{aone}
\begin{itemize}
    \item[(i)] There is an isomorphism of complex manifolds
    $\cS\isom \Pol_0$. 
    \item[(ii)] The flat co-ordinates for the integral affine structure on $\cS$ are functions of the form
    \[z_{\gamma}=\int_\gamma \sqrt{q(x)} \, dx,\]
    where $\gamma$ is a cycle on the Riemann surface of $\sqrt{q(x)}$. 
    
    \item[(iii)] The action of $s\in \bC$ on $\cS$ corresponds to the map  $a_k\mapsto e^{2s(m-k-2)/m} a_k$.
\end{itemize}
\end{thm}

Given a polynomial $q\in \Pol$,  we can consider the differential equation in the complex plane\[  y''(x) = q(x) y(x),\]
where primes denote differentiation with respect to $x$. 
The space  of solutions to this equation is a two-dimensional vector space $V\isom \bC^2$. For each $j\in \bZ_{m}$ we can consider the open sector \[S_j=\big\{z\in \bC:|\arg(z)-2\pi j/m|<\pi/m\big\}.\]
Sibuya \cite{Sib} showed that  there is a unique one-dimensional subspace  $\ell_j\subset V$ consisting of solutions which decay exponentially as $x\to \infty$ in this sector. 

Results of Fock and Goncharov \cite{FG0} imply that the cluster space $\cX_{\bC}$ can be identified with an open subset of framed local systems on a disc with $m$ marked points on the boundary. More explicitly,   it is proved in \cite[Section 3]{All} that 
\begin{equation}
    \label{late} 
\cX_{\bC}\isom \big\{p\colon \bZ_{m}\to \bP^1: \hash(\Im(p))>2 \text{ and } p(j)\neq p(j+1)\text{ for all $j\in \bZ_m$} \big\}\big/\PGL_2.\end{equation}

In Section \ref{secsurface} we use  work of Gupta and Mj \cite{GM} to prove the following result.

\begin{thm}
\label{atwo}
   \begin{itemize}
    \item[(i)] There is an isomorphism of complex manifolds
    $\cL\isom \Pol$.
\item[(ii)] Under the above identifications, the  map $\expl\colon \cL\to \cX_\bC$   sends a polynomial $q\in \Pol$ to the point of the quotient space \eqref{late}  defined by  $p(j)=\ell_j\in \bP(V)$.
    \end{itemize}
\end{thm}

We note that relations between moduli spaces of projective structures and log cluster transformations can already be seen in the preprint \cite{F} which appeared ten years before the invention of clusters!

\subsection{Cluster twistor space}

The cluster twistor space $\cZ$ will be constructed in the sequel \cite{sequel} to this paper. It comes equipped with a map  $\pi\colon \cZ\to \bC$. As with all the other spaces, it carries an action of the cluster modular group $\bG$, and this action preserves the map $\pi$.

 \begin{thm}
 \label{propsintro}
The cluster twistor space is a complex manifold $\cZ$ with the following properties:
\label{props}\begin{itemize}

\item[(i)] the map $\pi\colon Z\to \bC$ is a holomorphic submersion,

\item[(ii)] there is an action of  $\bC$  on $\cZ$ such that $\pi(s\cdot x)=e^{s}\cdot \pi(x)$ for all $s\in \bC$,

\item[(iii)] the action of $\pi i\in \bC$ coincides with the action of the DT element $T\in \bG$,

\item[(iv)] the fibre of $\pi$ over $0\in \bC$ is isomorphic to $\cS$,

\item[(v)] the fibres of $\pi$ over points of $\bC^*\subset \bC$ are isomorphic to $\cL$.
\end{itemize}
\end{thm}

The  motivation for our constructions comes from a line of research which aims to encode the DT invariants of a CY$_3$ triangulated category in a geometric structure on its space of stability conditions. This proceeds by considering a class of non-linear Riemann--Hilbert (RH) problems involving piecewise holomorphic maps $X\colon \bC^*\to (\bC^*)^n$  with prescribed jumps along a collection of rays defined by the DT invariants \cite{RHDT1,RHDT2}. In cases when these problems have well-behaved solutions, the end result is expected to be a complex \hk structure on the tangent bundle of the space of stability conditions \cite{Strachan,BJoy}. 

Penrose \cite{NLG}   showed that a complex \hk structure has an associated twistor space $\pi\colon \cZ\to \bP^1$, and  that moreover the \hk structure can  be reconstructed as the moduli space of certain sections of $\pi$ known as twistor lines. The reader can consult  \cite[Section 9.2]{GS} for a precise mathematical statement and detailed references.
The claim is then that the cluster twistor space of Theorem \ref{propsintro}  is the restriction of the  twistor space of the expected complex \hk structure to the subset $\bC\subset \bP^1$. Moreover the solutions to the RH problems describe the twistor lines in this twistor space. 

This picture is closely related to the work of Gaiotto, Moore and Neitzke \cite{GMN2} who explained that solutions to a similar class of RH problems can be interpreted  geometrically as  holomorphic maps $\bC^*\to \cX_{\bC}$  expressed in cluster charts which vary according to the argument of $\epsilon\in \bC^*$.  We are attempting to  go further and argue that the condition  imposed on the solution to the RH problem $X(\epsilon)$ in the limit as  $\epsilon\to 0$ can also be repackaged geometrically by requiring  that the holomorphic map  extends over $\epsilon=0$  to produce a section of the map $\pi\colon \cZ\to \bC$.

We view the  CY$_3$ categories associated to ADE quivers  as being the simplest non-trivial examples of CY$_3$ categories. In  these cases   all information about the DT invariants is encoded in the  combinatorics of the various cluster charts on the space $\cX_{\bC}$, and it is therefore reasonable to hope for a complete description of the twistor space in terms of cluster combinatorics. It is worth noting however that even in the very simplest case of the A$_2$ quiver, the twistor lines and the complex \hk structure are highly non-trivial, and are related to isomonodromic deformations and the Painlev{\'e} I equation \cite{Masoe}.

\subsubsection*{Plan of the paper}
We begin by recalling some basic facts from cluster theory. In Section \ref{sectwo} we introduce the  exchange graph, on which our entire treatment is based, and in Section \ref{secthree} we define the complex, real and tropical cluster spaces. Section \ref{secfour} introduces  the tropical fan, and  Section \ref{secfive} describes the specialisation of the  complex cluster  space to an algebraic torus.
In Section \ref{secdual} we introduce the DT element and discuss a duality  relating the exchange graphs of  a quiver and its opposite. 
Section \ref{secsix} contains some technical results about quotients  of the tropical fan.

The cluster stability space $\cS$ is  introduced in Section \ref{secseven} where we also prove parts (i)\,-\,(iii) of Theorem \ref{clstab}. Part (iv) is proved in Section \ref{cy3} after a brief summary of the  necessary background material on  CY$_3$ categorification of cluster combinatorics.

In Section \ref{tangent} we study the total space $\cT$ of the tangent bundle to the positive real cluster space $\cX_{\bR_+}$ and construct a $C^1$ local homeomorphism $\expt\colon \cT\to \cX_{\bC}$. This relies on a result about deformations of polyhedral fans which we defer to Appendix \ref{fan-def-section}.
In Section \ref{secnine} we define the log cluster space $\cL$ and the  map $\expl\colon \cL\to \cX_{\bC}$. We then use the results of Section \ref{tangent} to prove Theorem \ref{thisone}.  

Section \ref{secsurface} is devoted to the special case of the A$_n$ quiver. We relate the spaces $\cS$ and $\cL$  to moduli spaces of framed meromorphic quadratic differentials and projective structures respectively. In particular we prove Theorems \ref{aone} and \ref{atwo}. 

\subsubsection*{Acknowledgements}

This paper has been in the works for a long time. TB would like to acknowledge a very helpful email exchange with  Sasha Goncharov in January 2021.  He  also thanks the Royal Society for financial support in the form of a University Research Professorship. HR thanks the Norwegian Research Council for funding his Shape2030 grant.


\section{Exchange graphs}
\label{sectwo}

In the first four sections we review some standard background material on clusters and mutations originally due to Fock and Goncharov \cite{FG0,FG1}, and Fomin and Zelevinsky \cite{FZ1,FZ2}.  We shall take as our basic reference the recent monograph by Nakanishi \cite{Nak}. Our point-of-view  is strongly influenced by the connection with  CY$_3$ categories summarised in Section \ref{cy3}, but since we didn't want to make this material a pre-requisite, we have attempted instead to give a purely combinatorial treatment which nonetheless emphasises  those features which are most natural from the CY$_3$ point-of-view. In particular, our exchange matrices will be  skew-symmetric rather than skew-symmetrisable.
 
\subsection{Notation and conventions}
\label{notconv}
An \emph{oriented  graph}
  consists of a set of vertices $V$, a set of edges $I$, and source and target maps\todo{It is unfortunate that the letters $s$ and $t$ are already over-used.} $s,t \colon I\to V$.  We write  $i\colon a\to b\in I$ to mean that $i\in I$ with $s(i)=a$ and $t(i)=b$. 
 We denote by $I(a)=s^{-1}(a)\subset I$ the set of edges emanating from a given vertex $a\in V$. We say that a graph is \emph{finite} if the sets $V$ and $I$ are finite, \emph{locally-finite} if the sets $s^{-1}(a)$ and $t^{-1}(a)$ are  finite for all $a\in V$, and \emph{connected} if any two vertices can be connected by a {finite} sequence of edges. By viewing an unoriented edge as a pair of oppositely-oriented edges we  define an \emph{unoriented graph} to be an oriented graph equipped with the extra data of a fixed-point free involution $\epsilon\colon I\to I$ satisfying $t(i)=s(\epsilon(i))$. 
By an \emph{orientation} of an unoriented graph we mean a decomposition $I=I_+\sqcup I_-$ such that $\epsilon(I_+)=I_-$. 

By a \emph{quiver}  we mean a finite oriented graph. By the \emph{opposite} of a quiver $Q$ we mean the quiver $\widebar{Q}$ with the same underlying unoriented graph but equipped with the opposite orientation. A quiver $Q$ is said to be \emph{2-acyclic} if it has no loops or oriented 2-cycles.

Given  a finite set $I$, we consider the  free abelian group $\bZ_I=\left\{\sum_{i\in I} n_i \cdot e_i:n_i\in \bZ\right\}$   with its canonical basis $\{e_i:i\in I\}$.  We identify the dual group  $\Hom_{\bZ}(\bZ_I,\bZ)$ with the set $\bZ^I$ of functions $w\colon I\to \bZ$ by sending a homomorphism $\theta\colon \bZ_I\to \bZ$ to the function  defined by $w(i)=\theta(e_i)$.  We denote by $\{e_i^*:i\in I\}$ the dual basis of the group $\bZ^I$.  
We write $\bC(I)$ for the field of fractions of the group algebra $\bC[I]$ of the group $\bZ_I$, and   $x_i\in \bC(I)$ for the image of the basis element   $e_i\in \bZ_I$ under the canonical inclusion $\bZ_I\subset \bC(I)$. Thus $\bC(I)$  is the field of rational functions $\bC(x_i:i\in I)$ in indeterminates $x_i$ indexed by the elements $i\in I$. 

We define the real vector spaces $\bR_I$ and $\bR^I$ in the same way as the groups $\bZ_I$ and $\bZ^I$. By an \emph{orthant} in $\bR^I$ we mean a subset of the form $\{w\in \bR^I:\kappa(i)\cdot w(i)\geq 0 \ \forall i\in I\}$, where $\kappa\colon I\to \{\pm 1\}$ is some collection of signs. 
By a \emph{cone} in $\bR^I$ we always mean a polyhedral cone. We denote by $\<\sigma\>\subseteq \bR^I$  the linear subspace spanned by a cone $\sigma\subset \bR^I$.

\subsection{Exchange graphs}
\label{exch}

We base our treatment of cluster combinatorics on the following   slightly non-standard definitions. The relation to the usual terminology of cluster theory will be  explained in  the next subsection.  

\begin{definition}
\label{defn}
An \emph{exchange graph} $\bE$ is  a connected, locally-finite, unoriented graph $(V,I,s,t,\epsilon)$ as defined above, equipped with extra data $(v,\rho,\phi)$ 
as follows.
 
\begin{itemize}
\item[(E1)] \emph{Exchange matrices.} For each vertex $s\in V$ there is a collection of integers $\{v_{i,j}\in \bZ:i,j\in I(s)\}$ satisfying $v_{i,j}=-v_{j,i}$. We often write $v_{ij}=v_{i,j}$. \smallskip

\item[(E2)] \emph{Edge bijections.} For each edge  $i\colon s_1 \to s_2\in I$ there is a bijection $\rho_i\colon I(s_1)\to I(s_2)$ such that $\rho_i(i)=\epsilon(i)$ and $\rho_{\epsilon(i)}=\rho_i^{-1}$.\smallskip

\item[(E3)] \emph{Field isomorphisms.} For each pair of vertices $s_1,s_2\in V$ there is a $\bC$-linear field isomorphism $\phi(s_1,s_2)^*\colon \bC(I(s_2))\to \bC(I(s_1))$ satisfying  $\phi(s_1,s_3)^*=\phi(s_1,s_2)^*\circ \phi(s_2,s_3)^*$.\end{itemize}

This data should satisfy the following conditions:

\begin{itemize}
\item[(E4)] \emph{Mutation rules.} If   $i\colon s_1 \to s_2\in I$  is an edge then for all $j,k\in I(s_1)$
\begin{equation}
\label{form}v_{\rho_i(j),\rho_i(k)}=\begin{cases}     v_{jk}+v_{ji}\, v_{ik}&\text{ if }i\notin \{j,k\}\text{ and }v_{ji}\geq 0 \text{ and }v_{ik}\geq 0, \\ v_{jk}-v_{ji}\, v_{ik}&\text{ if }i\notin \{j,k\}\text{ and }v_{ji}\leq 0\text{ and } v_{ik}\leq 0,\\ v_{jk}&\text{ if }i\notin \{j,k\}\text{ and }v_{ji}\,v_{ik}\leq 0, \\ -v_{jk}&\text{ if }i\in \{j,k\}.  \end{cases}\end{equation}
\begin{equation}\label{glue}\hspace{-4.6em}\phi(s_1,s_2)^*(x_{\rho_i(j)})=\begin{cases} 
x_j\, (1+x_i^{+1})^{v_{ij}} &\text{ if } i\neq j\text{ and }v_{ij}\geq  0, \\ x_j\, (1+x_i^{-1})^{v_{ij}} &\text{ if } i\neq j\text{ and }v_{ij}<0, \\
x_{i}^{-1} &\text{ if }i=j.\end{cases}\end{equation}

\item[(E5)] \emph{Minimality}. Suppose we are given two vertices $s_1,s_2\in V$ and a bijection $h\colon I(s_1)\to I(s_2)$ such that  $\phi(s_1,s_2)^*(x_{h(i)})=x_{i}$  for all $i\in  I(s_1)$. Then in fact $s_1=s_2$ and $h$ is the identity. \end{itemize}
\end{definition}

Note that,  unlike the field isomorphisms in (E3), the edge bijections in (E2) are not required to satisfy a cocycle property. Thus, in general,  given a cycle $s=s_0\to s_1\to\cdots \to s_n=s$ in the underlying graph of $\bE$, the composition of the bijections $I(s_i)\to I(s_{i+1})$   defines a non-trivial permutation of  the set $I(s)$.

\begin{definition}
\label{eiso}
An \emph{isomorphism of exchange graphs} $g\colon \bE_1\to \bE_2$ is  an isomorphism of the underlying unoriented graphs which preserves the extra data specified in  (E1) -- (E3).
\end{definition}

 To spell this out, the isomorphism of unoriented graphs  is defined by a pair of bijections  $g\colon V_1\to V_2$ and $g\colon I_1\to I_2$ satisfying \[g(s(i))=s(g(i)), \qquad g(t(i))=t(g(i)), \qquad \epsilon(g(i))=g(\epsilon(i)),\]  for all edges $i\in I$. We then  further require that  
for all vertices $s\in V$ and edges $i,j\in I(s)$ \[v_{g(i),g(j)}=v_{i,j}, \qquad \rho_{g(i)}(g(j))=g(\rho_i(j)).\]
Connecting an arbitrary pair of vertices $s_1,s_2\in V$ by a chain of edges and applying the mutation rule (E4) it then follows that \[\phi(g(s_1),g(s_2))^*\circ g_{s_2}=g_{s_1}\circ \phi(s_1,s_2)^*\]
where for all vertices $s\in V$ the map $g_s\colon \bC(I(s))\to \bC(I(g(s)))$ is defined  by $g_s(x_i)=x_{g(i)}$.

\subsection{Relation to $Y$-patterns}

Let us briefly explain the relation of the above definition to the usual notions  of  cluster theory. 

Suppose we are given a $Y$-pattern as in \cite[Definition I.2.13]{Nak} with values in the universal semifield \cite[Example I.1.5(a)]{Nak} on $n$ generators.  We can then define an exchange graph  as in Definition \ref{defn} by quotienting the $n$-regular tree $\bT_n$  by the group of automorphisms of the graph $\bT_n$ which preserve the $Y$-pattern  up to permutation. More precisely, we quotient by   automorphisms  $g\colon \bT_n\to \bT_n$ such that for each vertex $t\in \bT_n$ the coefficient tuples corresponding to the vertices $t$ and $g(t)$ differ by a permutation. Note that by \cite[Theorem I.4.27]{Nak} any such automorphism preserves the entire cluster pattern up to permutation.

The variables $x_i$ appearing in (E4) are what are called coefficients in \cite{Nak} (and denoted $y_i$) and we set $v_{ij}=-b_{ij}$.  The edge bijections $\rho_i$ of (E2) are induced by the canonical identifications between the edges incident with any vertex of $\bT_n$. The minimality property (E5) holds because an easy argument shows\todo{Perhaps we should discuss this ...} that if the coefficient tuples corresponding to two vertices $t,t'\in \bT_n$ differ by a  permutation then there is an automorphism $g\colon \bT_n\to \bT_n$ satisfying $g(t)=t'$ which preserves the $Y$-pattern up to permutation.

In the opposite direction, suppose given an exchange graph  as in Definition \ref{defn}. For each vertex $s\in V$ we can consider the semifield $\bC_{\sf}(I(s))\subset \bC(I(s))$ consisting of rational functions with subtraction-free expressions \cite[Example I.1.5(a)]{Nak}.  Note that this is the universal semifield on generators $x_i$ indexed by $i\in I(s)$. The form of  the transformations \eqref{glue} ensures that the field isomorphisms $\phi(s_1,s_2)^*$ induce isomorphisms of  semifields $\phi(s_1,s_2)^*\colon \bC_{\sf}(I(s_2))\to \bC_{\sf}(I(s_1))$. Choosing a base vertex  and using the bijections $\rho_i$ to identify the edges incident at each vertex, we can view the underlying oriented graph of $\bE$ as a quotient of the $n$-regular tree $\bT_n$. Lifting the data to this tree we obtain a $Y$-pattern with coefficients in the universal semifield on $n$ generators.

 \subsection{Quivers and mutations}
 \label{qm}
 
 Given an exchange graph $\bE$ and a vertex $s\in V$,  there is an associated 2-acyclic quiver $Q(s)$ whose set of vertices is $I(s)$, and which has $\max(0,-v_{ij})$ edges with source $i$ and target $j$.
We say that an exchange graph \emph{includes} a given 2-acyclic quiver $Q$ if there is a vertex $s\in V$ such that  $Q(s)$ is isomorphic to $Q$. Two 2-acyclic quivers are said to be \emph{mutation-equivalent} if they are included in the same exchange graph.
Standard arguments\todo{... and this!} show that  every 2-acyclic quiver is included in an exchange graph, and that this exchange graph is unique up to isomorphism.
 Thus  isomorphism classes of exchange graphs are in bijection with mutation-equivalence classes of 2-acyclic quivers.
 
 An exchange graph  is said to be of \emph{finite type} if its underlying unoriented graph  is finite. We recall the following  result of Fomin and Zelevinsky (\cite{FZ2}, see also \cite[Theorem I.3.15]{Nak}).

 \begin{thm}
An exchange graph  is of finite type precisely if it  includes an orientation of  a  simply-laced  (but not necessarily connected) Dynkin diagram. \qed
 \end{thm}
  
Given an exchange graph $\bE$,   the \emph{opposite  exchange graph} $\widebar{\bE}$  is defined as follows \cite[II.2.2]{Nak}. We take the same underlying graph with the same edge bijections: $\widebar{V}=V$,  $\widebar{I}=I$ and $\widebar{\rho}_i=\rho_i$. We then set \begin{equation}
     \label{oppo}
\widebar{v}_{ij}=-v_{ij}, \qquad \bar{\phi}(s_1,s_2)^*=\iota^*(s_1)\circ \phi(s_1,s_2)^*\circ \iota^*(s_2), \end{equation}
 where for any vertex $s\in V$ the involution $\iota^*(s)$  of the field $\bC(I(s))$ is defined by $\iota^*(s)(x_i)=x_i^{-1}$. The mutation rules of (E4) are easily checked. Note that the quivers included in $\widebar{\bE}$ are precisely the opposites of those included in $\bE$.  Thus a quiver $Q$ is mutation-equivalent to its opposite  quiver $\widebar{Q}$ precisely if the corresponding exchange graphs $\bE$ and $\widebar{\bE}$ are isomorphic. 


\section{Cluster spaces}
\label{secthree}

{ For the rest of the paper we fix an exchange graph $\bE$ as in Definition \ref{defn}}, and we use the associated notation $(V,I,s,t,\epsilon)$ and $(v_{ij}, \rho_i, \phi(s_1,s_2)^*)$ without further comment. {From Section \ref{secsix} onwards we will assume that $\bE$ is of finite type.}
In this section we review the definitions of the various associated cluster spaces,  following Fock and Goncharov \cite{FG1}.


\subsection{Complex cluster space}
\label{ccs}

For each vertex $s\in V$ we consider the space $(\bC^*)^{I(s)}$ of $\bC$-valued points of the algebraic torus  with character lattice $\bZ_{I(s)}$ and function field $\bC(I(s))$. We view elements of $(\bC^*)^{I(s)}$ as maps $w\colon I(s)\to \bC^*$, and  the element $x_i\in \bC(I(s))$ then corresponds to the function $x_i\colon (\bC^*)^{I(s)}\to \bC^*$  defined by $x_i(w)=w(i)$. The field automorphisms $\phi(s_1,s_2)^*$ induce birational isomorphisms of tori  and hence partially-defined maps \[\phi_{\bC}(s_1,s_2)\colon (\bC^*)^{I(s_1)}\dashrightarrow (\bC^*)^{I(s_2)}.\] 
Given an edge  $i\colon s_1 \to s_2\in I$ and an element  $w\in (\bC^*)^{I(s_1)}$, the mutation formula \eqref{glue} gives
\begin{equation}\label{gluve}\phi_\bC(s_1,s_2)(w)(\rho_i(j))=\begin{cases} 
w(j)\, (1+w(i)^{+1})^{v_{ij}} &\text{ if } i\neq j\text{ and }v_{ij}\geq  0, \\ w(j)\, (1+w(i)^{-1})^{v_{ij}} &\text{ if } i\neq j\text{ and }v_{ij}<0, \\
w(i)^{-1} &\text{ if }i=j.\end{cases}\end{equation}

We shall view the complex cluster  space  $\cX_\bC$ as a possibly non-Hausdorff complex manifold obtained by quotienting the union $\bigcup_{s\in V} (\bC^*)^{I(s)}$ by the equivalence relation \[w_1\in (\bC^*)^{I(s_1)}\sim w_2\in (\bC^*)^{I(s_2)}\iff  \phi_\bC(s_1,s_2)\text{ is regular at $w_1$ and  $\phi_\bC(s_1,s_2)(w_1)=w_2$}.\]
There are open embeddings $\phi_\bC(s)^{-1}\colon (\bC^*)^{I(s)}\to \cX_\bC$ and equalities of partially-defined maps $\phi_\bC(s_2)=\phi_\bC(s_1,s_2)\circ \phi_\bC(s_1)$. 

For each vertex $s\in V$ we can define a skew-symmetric form
\begin{equation}
    \label{skewforms}\<-,-\>_s\colon \bZ_{I(s)}\times \bZ_{I(s)}\to \bZ, \qquad \<e_j,e_k\>=v_{ij}.
\end{equation}
This induces a translation-invariant  Poisson structure $\{-,-\}_s$ on the torus $(\bC^*)^{I(s)}$ satisfying
\[\{x_j,x_k\}_s=2\pi i\cdot v_{jk}\cdot x_jx_k\]
 for all $j,k\in I(s)$. A computation using \eqref{form} and \eqref{glue} shows that these  Poisson structures are intertwined by the maps $\phi_\bC(s_1,s_2)$, and therefore glue to give a holomorphic Poisson structure $\{-,-\}$ on $\cX_{\bC}$. 

\subsection{Positive real cluster space}
\label{posreal}

For each pair of vertices $s_1,s_2\in V$ the partially-defined map $\phi_{\bC}(s_1,s_2)$ restricts to a diffeomorphism
\[\phi_{\bR_+}(s_1,s_2)\colon \bR_+^{I(s_1)}\to \bR_+^{I(s_2)},\]
since this is clearly true for each of the maps \eqref{gluve}. 
We define the real positive cluster variety $\cX_{\bR_+}$ by gluing together the sets $\bR_+^{I(s)}$  using these maps. It is a smooth manifold with a collection of diffeomorphisms $\phi_{\bR_+}(s)\colon \cX_{\bR_+}\to \bR_+^{I(s)}$ satisfying $\phi_{\bR_+}(s_2)=\phi_{\bR_+}(s_1,s_2)\circ\phi_{\bR_+}(s_1)$. Via the embeddings $\bR_+^{I(s)}\subset (\bC^*)^{I(s)}$ we can view $\cX_{\bR_+}$ as a closed submanifold of $\cX_{\bC}$.
 
For each vertex $s\in V $ there are mutually inverse  homeomorphisms \[\exp\colon \bR^{I(s)} \to \bR_+^{I(s)}, \qquad \log\colon \bR_+^{I(s)}\to  \bR^{I(s)}\] defined by the componentwise exponential and logarithm maps.
We set \[\phi_{\bR}(s)=\log\circ \,\phi_{\bR_+}(s)\colon \cX_{\bR_+}\to \bR^{I(s)},\]
and define transition functions $\phi_{\bR}(s_1,s_2)=\phi_{\bR}(s_2)\circ \phi_{\bR}(s_1)^{-1}$.
Given an edge    $i\colon s_1 \to s_2\in I $ and an element $w\in \bR^{I(s_1)}$, we have
\begin{equation}
\label{log_glue2}\phi_{\bR}(s_1,s_2)(w)(\rho_i(j))=\begin{cases} w(j)+v_{ij}\cdot \log\big(1+e^{+w(i)}\big) &\text{ if } i\neq j\text{ and }v_{ij}\geq 0, \\
w(j)+v_{ij}\cdot\log \big(1+e^{-w(i)}\big) &\text{ if } i\neq j\text{ and }v_{ij}<0, \\
-w(i) &\text{ if }i=j.\end{cases}
\end{equation}

\subsection{Tropical cluster space}
\label{tropclsp}

Fix an element $r\in \bR_+$ and for each vertex $s\in V$ denote by $m_r\colon \bR^{I(s)}\to \bR^{I(s)}$ the operation of componentwise multiplication by $r$. We obtain a new system of charts \[\phi_{\bR}^r(s)=m_r\circ \phi_{\bR}(s)\colon \cX_{\bR_+}\to \bR^{I(s)}\] with  corresponding transition functions \[\phi^r_{\bR}(s_1,s_2)=\phi^r_{\bR}(s_2)\circ \phi^r_{\bR}(s_1)^{-1}=m_r\circ \phi_{\bR}(s_1,s_2)\circ m_r^{-1}.\]

\begin{lemma}
\label{feb}
For each pair of vertices $s_1,s_2\in V$,
\begin{itemize}
\item[(i)]  the map $\phi^r_{\bR}(s_1,s_2)\colon \bR^{I(s_1)}\to \bR^{I(s_2)}$ has a pointwise limit as $r\to 0$,
\item[(ii)] the limit defines a piecewise linear (PL) isomorphism $\phi_{\bR^t}(s_1,s_2)\colon \bR^{I(s_1)}\to \bR^{I(s_2)}$,
\item[(iii)] the map $\phi_{\bR^t}(s_1,s_2)$ preserves the  integral lattices and multiplication by $r\in \bR_+$: \[\phi_{\bR^t}(s_1,s_2)(\bZ^{I(s_1)})= \bZ^{I(s_2)}, \qquad m_r\circ \phi_{\bR^t}(s_1,s_2)=\phi_{\bR^t}(s_1,s_2)\circ m_r.\]\end{itemize}
\end{lemma}

\begin{proof}
Since any pair of vertices $s_1,s_2\in V $ can be connected by a finite sequence of edges it is enough to consider the case of a single edge   $i\colon s_1 \to s_2\in I $. Given an element $w\in  \bR^{I(s_1)}$, we have 
\begin{equation}
\label{log_glue3}\phi^r_{\bR}(s_1,s_2)(w)(\rho_i(j))=\begin{cases} w(j)+v_{ij}\cdot r\cdot \log\big(1+e^{+w(i)/r}\big) &\text{ if } i\neq j\text{ and }v_{ij}\geq 0, \\
w(j)+v_{ij}\cdot r\cdot \log \big(1+e^{-w(i)/r}\big) &\text{ if } i\neq j\text{ and }v_{ij}<0, \\
-w(i) &\text{ if }i=j.\end{cases}\end{equation}
This indeed has a well-defined limit as $r\to 0$ and we find that
\begin{equation}
\label{log_glue4}\phi_{\bR^t}(s_1,s_2)(w)(\rho_i(j))=\begin{cases} 
w(j)+v_{ij} \cdot \max(0,+w(i)) &\text{ if } i\neq j\text{ and }v_{ij}\geq  0, \\ w(j)+v_{ij} \cdot \max(0,-w(i)) &\text{ if } i\neq j\text{ and }v_{ij}<0, \\
-w(i) &\text{ if }i=j.\end{cases}\end{equation}
The claims (ii) and (iii) follow immediately by inspecting this formula.
\end{proof}

The tropical cluster space $\cX_{\bR^{t}}$ is obtained by gluing copies of $\bR^{I(s)}$ using the PL  isomorphisms $\phi_{\bR^t}(s_1,s_2)$.  It is a PL manifold with canonical PL isomorphisms $\phi_{\bR^t}(s)\colon  \cX_{\bR^{t}}\to  \bR^{I(s)}$ satisfying $\phi_{\bR^t}(s_2)=\phi_{\bR^t}(s_1,s_2)\circ \phi_{\bR^t}(s_1)$. Lemma \ref{feb}\,(iii) shows that there is a well-defined subset $\cX_{\bZ^{t}}\subset \cX_{\bR^t}$ of integral points of the tropical cluster space. For each vertex $s\in V$, the chart $\phi(s)$ restricts to a bijection $\phi(s)\colon \cX_{\bZ^{t}}\to \bZ^{I(s)}$.

\subsection{Cluster modular group}
\label{clmogr}

We define the notion of an automorphism of an exchange graph $\bE$ by specialising Definition \ref{eiso} to the case $\bE_1=\bE=\bE_2$.

\begin{defn}\label{defn2}
    The \emph{cluster modular group} $\bG=\bG(\bE)$ associated to an exchange graph $\bE$ is the group of automorphisms of $\bE$.
\end{defn}

 The group $\bG$ acts on all the cluster spaces we have defined so far and also on the ones defined below. The form of this action  is always the same, so to avoid repetition we will explain it  only in the case of the tropical cluster space $\cX_{\bR^t}$. Given $g\in \bG$ and $x\in \cX_{\bR^t}$, we define $g(x)\in \cX_{\bR^t}$ by the condition
\begin{equation}\label{sample}\phi_{\bR^t}(g(s))(g(x))(g(i))=\phi_{\bR^t}(s)(x)(i),\end{equation}
for all vertices $s\in V$ and edges $i\in I(s)$. Thus for each vertex $s\in V$, there is a  commuting diagram

\begin{equation}
\label{diagramm}
\begin{gathered}
\xymatrix@C=2em{  \cX_{\bR^t}\ar[d]_{\phi_{\bR^t}(s)}\ar[rr]^{g} &&  \cX_{\bR^t}\ar[d]^{\phi_{\bR^t}(g(s))}  \\
   \bR^{I(a)} \ar[rr]^{g}   && \bR^{I(g(s))}}\end{gathered}
 \end{equation}
where the bottom arrow is induced by the bijection  $g\colon I(s)\to I(g(s))$.


\section{Sign coherence and the tropical fan}
\label{secfour}

In this section we again fix an exchange graph $\bE$ as in Definition \ref{defn} and review some further standard material from cluster theory, namely  sign coherence  and the tropical fan. 

\subsection{Sign coherence}

The following fundamental result  is known in the cluster literature as sign coherence.  

\begin{thm}
\label{coh}
    Take two vertices $a,s\in V$. Then  there is a single orthant in $\bR^{I(s)}$ which contains 
    the element $g_j=\phi_{\bR^t}(a,s)(e_j^*)$ for  all $j\in I(a)$.\qed
\end{thm}

We note in passing that although a purely combinatorial proof of this result is difficult, in the approach via  CY$_3$ categorification reviewed in Section \ref{cy3} it comes essentially for free.

\begin{remark}
\label{remproof}
    Let us relate Theorem \ref{coh} to the usual sign coherence statement in cluster theory.   Note first that the formula \eqref{log_glue4} describes the  mutation of coefficients with values in the tropical semifield $\bZ_t$.  Consider as in \cite[Section I.4.1]{Nak} the $Y$-system with  principal coefficients at the vertex $a\in V$.  For each $j\in I(a)$, the power of the element $y_j$ in the coefficient at a given vertex $s$ transforms under mutation as would a coefficient with values in $\bZ_t$. Using the definition of the $C$-matrix \cite[Definition I.4.3]{Nak} it follows that each vector $g_j$ appears as a row in the matrix $\widebar{C}_s$ associated to the opposite exchange graph $\widebar{\bE}$. The opposite comes from the mismatch between the maximum appearing in \eqref{log_glue4} and the minimum convention for the tropical semifield used in \cite[Example I.1.4 (b)]{Nak}. This statement also follows directly from \cite[Proposition I.4.4]{Nak}. The claim that all the vectors $g_j$  lie in a  single orthant  then follows immediately from the sign coherence property for $c$-vectors \cite[Theorem I.4.21]{Nak}. To explain our notation, let us also note that by \cite[Proposition II.2.7]{Nak} the vectors $g_j$ are the columns of the matrix $G_s$ associated to the original exchange graph $\bE$. 
\end{remark}

\subsection{Tropical cones}
For each vertex $a\in V$, we define the subset \begin{equation}
    \label{max}
\sigma^a=\phi_{\bR^t}(a)^{-1}(\bR^{I(a)}_{\geq 0})\subset \cX_{\bR^t}.\end{equation}
A subset $\sigma\subset \cX_{\bR^t}$ will be called a \emph{maximal tropical cone} if it is of the form $\sigma=\sigma^a$ for some vertex $a\in V$.  The terminology is justified by Theorem \ref{coherence} below. Given  a maximal   tropical cone  $\sigma\subset \cX_{\bR^t}$ and a vertex $s\in V$, we define
 \begin{equation}
 \label{cone}\sigma(s)= \phi_{\bR^t}(s)(\sigma)\subset \bR^{I(s)}.
 \end{equation}

\begin{thm}
\label{coherence}
Fix a maximal  tropical  cone  $\sigma\subset \cX_{\bR^t}$.  Then
\begin{itemize}
\item[(i)] for each vertex $s\in V$ the subset $\sigma(s)\subset \bR^{I(s)}$ is a rational simplicial cone of maximal dimension,
    \item[(ii)]
for any pair  of vertices $s_1,s_2\in V$  the  bijection \[\phi_{\bR^t}(s_1,s_2)\colon \sigma(s_1)\to \sigma(s_2)\] is $\bR_+$-linear and therefore takes faces to faces,
\item[(iii)] for each vertex $s\in V$ the subset $\sigma(s)\subset \bR^{I(s)}$ is contained in a single orthant.
\end{itemize}
\end{thm}

 \begin{proof}
 We can write $\sigma=\sigma^a$ for some vertex $a\in V$ and then $\sigma(s)$ is the image of the positive orthant $\sigma^a(a)=\bR^{I(a)}_{\geq 0}\subset \bR^{I(a)}$ under the map $\phi_{\bR^t}(a,s)$.
We can connect the vertices $a,s\in V$ by a finite chain of edges $a=s_1\to \cdots \to s_n=s$. The PL isomorphism \eqref{log_glue4} corresponding to a simple mutation is clearly linear away from the coordinate hyperplanes.   Using  induction and Theorem \ref{coh}  it follows that for all $1\leq i\leq n$ the restriction of  $\phi_{\bR^t}(a,s_i)$ to the positive orthant  $\sigma^a(a)$ is $\bR_+$-linear and has image contained in a single orthant. This implies (iii) and shows that $\sigma^a(s)$ is the cone spanned by the vectors $\phi_{\bR^t}(a,s)(e_j^*)$.  Part (i)  follows since $\phi_{\bR^t}(a,s_i)$ is a bijection and preserves the integral lattices, and part (ii)  is immediate on writing $\phi_{\bR^t}(s_1,s_2)=\phi_{\bR^t}(a,s_2)\circ \phi_{\bR^t}(a,s_1)^{-1}$.
\end{proof}

Consider a maximal  tropical  cone  $\sigma\subset \cX_{\bR^t}$ and an edge $i\colon s_1\to s_2\in I$. Theorem \ref{coherence}(iii) implies that    there is a unique sign $\kappa_{\sigma}(i)\in \{\pm 1\}$ such that \[w\in \sigma(s_1)\implies \kappa_{\sigma}(i)\cdot w(i)\geq 0.\]
    This sign  is called the \emph{tropical sign} of the edge $i\in I$ with respect to  $\sigma$.
(Strictly speaking, following Remark \ref{remproof}, this is the tropical sign in the sense of \cite[Definition II.2.1]{Nak} for the opposite exchange graph $\widebar{\bE}$).

\begin{remark}
It is easy to deduce from \eqref{log_glue4} and the relation $\rho_i(i)=\epsilon(i)$ that $\kappa_{\sigma}(\epsilon(i))=-\kappa_{\sigma}(i)$. Thus, having chosen  $\sigma$, we obtain an orientation $I=I_+\sqcup I_-$ of the unoriented graph underlying $\bE$ by declaring $I_{\pm}\subset I$ to be the subset of edges $i\in I$ with tropical sign $\kappa_\sigma(i)=\pm 1$.  
\end{remark}

\subsection{Tropical fan}

Let $\sigma\subset \cX_{\bR^t}$ be a maximal  tropical  cone. We say that a subset $\tau\subseteq \sigma$ is  a \emph{face} of $\sigma$ if for some  vertex $s\in V$ the subset $\tau(s)=\phi_{\bR^t}(s)(\tau)$ is a face of the cone $\sigma(s)=\phi_{\bR^t}(s)(\sigma)$. Theorem \ref{coherence} (ii) shows that this condition is independent of the choice of  vertex $s\in V$.  We say that  a subset $\tau\subset \cX_{\bR^t}$ is a \emph{tropical cone} if it is a face  of a maximal  tropical  cone. 

\begin{thm}
\label{fan}
\begin{itemize}
\item[(i)] If $a,b\in V$ then $a\neq b$ implies that $\sigma^a\neq \sigma^b$.
\item[(ii)] The intersection of any two  tropical  cones in $\cX_{\bR^t}$ is a face of each.
\item[(iii)] If $\bE$ is of finite type then $\cX_{\bR^t}=\bigcup_{a\in V} \sigma^a$.\end{itemize}
\end{thm}

\begin{proof}
Part (i) follows from \cite[Theorem I.4.27]{Nak} and our definition of the exchange graph.   

Assume for a contradiction that (ii) is false. Then we can find  tropical cones $\tau_1,\tau_2$,  such that the intersection $\tau_1\cap\tau_2$  is not a face of $\tau_1$. Let us take a vertex $s\in V$ such that $\tau_2\subseteq \sigma^s$ and work in the chart $\phi(s)$. Recall that a subset $F\subseteq \tau_1(s)$ is a face  precisely if $x,y\in \tau_1(s)$ and $x+y\in F$ implies $x,y\in F$. 
Thus we can find elements $x,y\in \tau_1(s)$ with $x+y\in \tau_2(s)$ but $x\not\in \tau_2(s)$. Since $\tau_2(s)$ is a face of the cone $\sigma^s(s)=\bR_{\geq 0}^{I(s)}$ this implies that there is an $i\in I(s)$ such that either $x(i)+y(i)\geq0$ and   $x(i)<0$, or $x(i)+y(i)= 0$ and $x(i)\neq 0$. Either way, we conclude that $x(i)$ and $y(i)$ have opposite signs. But then $\tau_1(s)$  is not contained in a single orthant, which contradicts    Theorem~\ref{coherence} (iii).
 
Part (iii) follows from a topological argument. Fix a vertex $s\in V$ and consider the images $\tau(s)=\phi_{\bR^t}(s)(\tau)\subset \bR^{I(s)}$ of all   tropical cones $\tau$. The union of these cones defines a closed subset $C\subseteq \bR^{I(s)}$, and the union of all cones of codimension $\geq 2$ defines a closed subset $C'\subset C$. Consider the complementary open subsets  $U=\bR^{I(s)}\,\setminus\, C$ and $U'=\bR^{I(s)}\,\setminus\, C'$. Then $U'$ is  path-connected, and coincides with  the union of $U$ with the relative interiors of all cones of codimension $\leq 1$.  In particular, the intersection $C\cap U'$ is non-empty since it contains the interior of any maximal cone, for example the positive orthant. 

Suppose for a contradiction that $U$ is non-empty. Then we can  take a point $u\in U$ and join it by a continuous path in $U'$ to a point $w\in C\cap U'$. Any such  path must necessarily meet the boundary of $C$.  But in fact $\partial C\cap U'$ is empty because by \eqref{log_glue4}, the relative interior of every codimension 1 cone is contained in an open set given by itself together with the interiors of the adjacent maximal cones.
\end{proof}

 We call the set $\Sigma$ of  tropical  cones in $\cX_{\bR^t}$  the \emph{tropical fan}. The map $a\mapsto \sigma^a$ defines a  bijection $V\to \Sigma$.  Each PL chart $\phi_{\bR}(s)\colon \cX_{\bR_+}\to \bR^{I(s)}$ maps the  cones of $\Sigma$ bijectively onto the cones of a  simplicial fan $\Sigma(s)$ in the vector space $\bR^{I(s)}$. In the case that   $\bE$ is of finite type this fan is moreover complete.
 
 \begin{remark}\label{notquitefan}The transition function $\phi_{\bR}(s_1,s_2)$ maps the fan $\Sigma(s_1)$ to the fan $\Sigma(s_2)$ by a PL map taking cones to cones. This does not imply that the fans $\Sigma(s_i)$ are isomorphic in the usual sense since a  morphism of fans in toric geometry is required to be induced by a linear map. Nonetheless, all  combinatorial properties  of cones and their faces are preserved by the maps $\phi_{\bR}(s_1,s_2)$, and so for these purposes we can treat $\Sigma$ exactly as if it were a fan in a vector space.
\end{remark}


\section{Specialisation of the cluster spaces}
\label{secfive}
We again fix an exchange graph $\bE$ as in Definition \ref{defn}. In this section we recall how the sign coherence property of Theorem \ref{coh} leads to a deformation of the cluster space $\cX_{\bC}$ whose central fibre is an algebraic torus $(\bC^*)^n$. Throughout this section we fix a vertex $a\in V$ and write $\sigma=\sigma^a$ for the corresponding maximal  tropical cone in $\cX_{\bR^t}$. Everything we do will depend on this choice, but we often suppress this dependence from the notation. The deformation of the space $\cX_{\bC}$ depends also on a choice of element $m\in \bZ_{>0}^{I(a)}$.

\subsection{Linear maps}
\label{linmaps}

Given a pair of  vertices $s_1,s_2\in V$, there are corresponding cones $\sigma(s_i)\subset \bR^{I(s_i)}$. We define 
\begin{equation}
\label{linear2}\mu^{\sigma}(s_1,s_2)\colon \bR^{I(s_1)}\to \bR^{I(s_2)}\end{equation}
to be the unique linear map extending $\phi_{\bR^t}(s_1,s_2)\colon \sigma(s_1)\to \sigma(s_2)$. This is well-defined by Theorem~\ref{coherence}. Given three vertices $s_1,s_2,s_3\in V$, we have
\begin{equation}
    \label{relly}
\mu^\sigma(s_1,s_3)=\mu^\sigma(s_2,s_3)\circ \mu^\sigma(s_1,s_2).\end{equation}

Lemma~\ref{feb}\,(iii) shows that \eqref{linear2}  restricts to an isomorphism
\begin{equation}\label{linz}\mu^{\sigma}(s_1,s_2)\colon \bZ^{I(s_1)}\to \bZ^{I(s_2)}.\end{equation}
We shall also frequently use the induced maps 
\begin{equation}\label{linz2}\mu^{\sigma}(s_1,s_2)\colon \bC^{I(s_1)}\to \bC^{I(s_2)}, \qquad \mu^{\sigma}(s_1,s_2)\colon (\bC^*)^{I(s_1)}\to (\bC^*)^{I(s_2)}.\end{equation}
obtained by tensoring \eqref{linz} with  $\bC$ and $\bC^*$ respectively.

The next result computes the map \eqref{linear2} in the case of a simple mutation.

\begin{lemma}
For any edge $i\colon s_1 \to s_2\in I $ and  any $w\in  \bR^{I(s_1)}$ 
\begin{equation}
\label{log_glue5}\mu^{\sigma}(s_1,s_2)(w)(\rho_i(j))=\begin{cases} w(j)+\kappa_\sigma(i) \cdot v_{ij}\cdot w(i) &\text{ if } i\neq j\text{ and }\kappa_\sigma(i) \cdot v_{ij}\geq 0, \\
w(j) &\text{ if } i\neq j\text{ and }\kappa_\sigma(i) \cdot v_{ij}<0, \\
-w(i) &\text{ if }i=j.\end{cases}\end{equation}
\end{lemma}

\begin{proof}
The formula \eqref{log_glue5} certainly defines a linear map $\bR^{I(s_1)}\to \bR^{I(s_2)}$ so we must just show that it coincides with \eqref{log_glue4} when restricted to the cone $\sigma(s_1)\subset \bR^{I(s_1)}$. This follows because by definition of the tropical sign, when $w\in \sigma(s_1)$ we have $\kappa_\sigma(i)\cdot w(i)\geq 0$ for all $i\in I(s_1)$.
\end{proof}

Recall from \eqref{skewforms} the definition of the skew-symmetric forms $\<-,-\>_s$. 

\begin{lemma}
\label{linear-poisson}
For  any pair of vertices $s_1,s_2\in V $ the the dual map to \eqref{linz},
\begin{equation}
    \label{dualmap}
\mu^{\sigma}(s_1,s_2)^*\colon \bZ_{I(s_2)}\to \bZ_{I(s_1)},
\end{equation}
intertwines  the forms $\<-,-\>_{s_2} $ and  $\<-,-\>_{s_1}$.
\end{lemma}

\begin{proof} It is enough to consider the case when $s_1,s_2$ are connected by a single  edge $i\colon s_1\to s_2\in I $. The map \eqref{dualmap}  is then given by the formula
\begin{equation}
\label{log_glue6}\mu^{\sigma}(s_1,s_2)^*(e_{\rho_i(j)})=\begin{cases} e_j+\kappa_\sigma(i) \cdot v_{ij}\cdot e_i &\text{ if } i\neq j\text{ and }\kappa_\sigma(i)\cdot v_{ij}\geq 0, \\
e_j &\text{ if } i\neq j\text{ and }\kappa_\sigma(i)\cdot v_{ij}<0, \\
-e_i &\text{ if }i=j,\end{cases}\end{equation}
and the result follows by a  direct computation using \eqref{form}.\end{proof}

\subsection{Separation formula} 

For each vertex $s\in V $ we define a  birational chart 
\begin{equation}
\psi_{\bC}(s)=\mu^{\sigma}(s,a)\circ \phi_{\bC}(s)\colon \cX_{\bC}\dashrightarrow (\bC^*)^{I(a)}.
\label{psi-chart}
 \end{equation}
Given two vertices $s_1,s_2\in V $, the transition function $\psi_{\bC}(s_1,s_2)=\psi_{\bC}(s_2)\circ \psi_{\bC}(s_1)^{-1}$ is a birational automorphism  of $(\bC^*)^{I(a)}$ fitting into the following commutative square
 \begin{equation}
\begin{gathered}\label{biggy}
\xymatrix@C=1.8em{  (\bC^*)^{I(a)}\ar[d]_{\mu^{\sigma}(a,s_1)}\ar@{-->}[rr]^{\psi_{\bC}(s_1,s_2)} && (\bC^*)^{I(a)}\ar[d]^{\mu^{\sigma}(a,s_2)}  \\
   (\bC^*)^{I(s_1)} \ar@{-->}[rr]^{\phi_{\bC}(s_1,s_2)}   && (\bC^*)^{I(s_2)}}\end{gathered}
 \end{equation}

The following result computes the map $\psi_{\bC}(s_1,s_2)^*$ in the case of a simple mutation.

\begin{lemma}
\label{scourie}
Take an edge $i\colon s_1\to s_2\in I $ and set $d=\mu^{\sigma}(a,s_1)^*(e_i)\in \bZ_{I(a)}$. 
 Then for any $n\in \bZ_{I(a)}$ 
\begin{equation}
\label{neater}\psi_{\bC}(s_1,s_2)^*(x_n)=x_n\big(1+x_{d}^{-\kappa_\sigma(i)}\big)^{\<d,n\>_a}.
\end{equation}
\end{lemma}

\begin{proof}
Consider the automorphism $f$ of the field $\bC(I(s_1))$ defined by the composition
\[f=(\mu^{\sigma}(a,s_1)^*)^{-1}\circ \psi_{\bC}(s_1,s_2)^*\circ \mu^{\sigma}(a,s_1)^*=\phi_{\bC}(s_1,s_2)^*\circ \mu^{\sigma}(s_2,s_1)^*.\]
Computing the inverse of \eqref{log_glue6} gives
\[\mu^{\sigma}(s_2,s_1)^*(e_j)=\begin{cases} e_{\rho_i(j)}+ \kappa_\sigma(i)\cdot v_{ij}\cdot e_{\rho_i(i)} &\text{ if } i\neq j\text{ and }\kappa_\sigma(i)\cdot v_{ij}\geq 0, \\
e_{\rho_i(j)}  &\text{ if } i\neq j\text{ and }\kappa_\sigma(i)\cdot v_{ij}<0, \\
-e_{\rho_i(i)} &\text{ if }i=j.\end{cases}\]
Composing the corresponding transformation on monomials  with the formula \eqref{glue} we obtain 
\begin{equation*}f(x_j)=x_j (1+x_i^{-\kappa_\sigma(i)})^{v_{ij}}\end{equation*}
for all $j\in I(s_1)$. The result follows by conjugating by $\mu^\sigma(a,s_1)$ and using Lemma \ref{linear-poisson}.
\end{proof}

\subsection{Torus specialisation}
\label{tordeg-section}

Let us now fix an element   $m\in \bZ_{> 0}^{I(a)}$.
For each $t\in \bC^*$ the map
\begin{equation}
\label{alphat}
\alpha^*_t\colon \bC(I(a))\to \bC(I(a)),\qquad x_i\mapsto t^{m(i)}\, x_i
\end{equation}
defines an automorphism of the field $\bC(I(a))$. Define the composite automorphism
\begin{equation}\label{psi-t-chart}
          \psi_{\bC}(s_1,s_2)_t^* = (\alpha_t^*)^{-1}\circ\psi_{\bC}(s_1,s_2)^*\circ \alpha^*_t.   
\end{equation}

 Suppose we are given an edge $i\colon s_1\to s_2\in I$. Replacing $i$ by its opposite edge $\epsilon(i)$  if necessary, we can assume that  $\kappa_\sigma(i)=1$. Then \eqref{neater} gives
\begin{equation}
\label{neaterer}
\psi_{\bC}(s_1,s_2)_t^*(x_n)=x_n\left(1+t^{m(d)} x^{-1}_{d}\right)^{\<d,n\>_a}.
\end{equation}

We claim that the the power  $m(d)$  of $t$ appearing in \eqref{neaterer} is always positive. Indeed,  $m(d)$ is given by the natural pairing  between $m\in \bZ_{> 0}^{I(a)}$ and $ d=\mu^{\sigma}(a,s_1)^*(e_i)\in \bZ_{I(a)}$. Thus  
we can write $m(d)=\mu^\sigma(a,s_1)(m)(i)$. But since $\sigma=\sigma^a$ and $m\in \sigma^a(a)$, we have  $\mu^\sigma(a,s_1)(m)=\phi_{\bR^t}(a,s_1)(m)$, and this element therefore lies in interior of the cone $\sigma(s_1)\subset \bR^{I(s_1)}$. The definition of the tropical sign and the assumption that   $\kappa_\sigma(i)=1$ then gives $m(d)> 0$.

It now follows that for any $s_1,s_2\in V$ the map $\psi_{\bC}(s_1,s_2)$ has a well-defined limit as $t\to 0$, and that this limit is the identity.
In algebraic terms, the formula \eqref{alphat} defines an automorphism of the algebra $\bC(I(a))\tensor_{\bC} \bC[t,t^{-1}]$. By extension of scalars, the map 
$\psi_{\bC}(s_1,s_2)$ also defines  an automorphism of this algebra. The non-trivial statement is then that  the composite \eqref{psi-t-chart}  restricts to an automorphism of the subalgebra $\bC(I(a))\tensor_{\bC}\bC[t]$.
 Specialising to $t=1$ gives back the original automorphism $\psi_{\bC}(s_1,s_2)^*$, whereas specialising to $t=0$ gives the identity.

For each pair of vertices $s_1,s_2\in V$ we define a field isomorphism  $\phi_\bC(s_1,s_2)^*_t$
via the diagram
\begin{equation}
\begin{gathered}\label{biggy3}
\xymatrix@C=1.8em{  \bC(I(a)) && \bC(I(a))\ar[ll]_{\ \psi_{\bC}(s_1,s_2)_t^*}  \\
   \bC(I(s_1)) \ar[u]^{\mu^{\sigma}(a,s_1)^*}
   && \bC(I(s_2))\ar[ll]_{\ \phi_{\bC}(s_1,s_2)_t^*}\ar[u]_{\mu^{\sigma}(a,s_2)^*}}\end{gathered}
 \end{equation}
Then $\phi_{\bC}(s_1,s_2)^*_t$ have a well-defined limit at $t=0$ and this  coincides with $\mu^\sigma(s_1,s_2)^*$.

For a single edge $i\colon s_1\to s_2$ with $\kappa_\sigma(i)=1$ a direct calculation using \eqref{glue} and \eqref{log_glue5} shows that
\begin{equation}
\label{log_glue7}\phi_{\bC}(s_1,s_2)_t^*(x)(\rho_i(j))=
\begin{cases} x(j)\left(t^{r(i)}+x(i)\right)^ {v_{ij}} &\text{ if } i\neq j\text{ and }v_{ij}\geq 0, \\
x(j) \left(1+ t^{r(i)}\, x(i)^{-1}\right)^ {v_{ij}}&\text{ if } i\neq j\text{ and }v_{ij}\leq 0, \\
x(i)^{-1} &\text{ if }i=j,\end{cases}\end{equation}
where $r(i)=\mu^\sigma(a,s)(m)(i)>0$.

The isomorphisms $\phi_{\bC}(s_1,s_2)^*_t$  induce birational maps of algebraic tori \[\phi(s_i,s_j)_t\colon (\bC^*)^{I(s_1)}\dashrightarrow (\bC^*)^{I(s_2)}.\] When $t=0$ these maps are biregular.
Working exactly as in Section \ref{ccs} we can use the maps $\phi(s_i,s_j)_t$ to glue together the algebraic tori $(\bC^*)^{I(s_i)}$. This results in a possibly non-Hausdorff complex manifold $\cX_{\bC,t}$ with a collection of charts $\phi_\bC(s)_t\colon \cX_{\bC,t}\dashrightarrow (\bC^*)^{I(s)}$. When $t=0$ each of these charts is a biregular  isomorphism.


\section{Duality and the DT element}
\label{secdual}

Let $\bE$ be an exchange graph, and recall from Section \ref{qm} the definition of the opposite exchange graph $\widebar{\bE}$.   Note in particular  that the underlying unoriented graphs $(V,I,s,t,\epsilon)$ of $\widebar{\bE}$ is the same as that of $\bE$. We collect here several results which relate these two exchange graphs.  This leads naturally to the definition of a DT element of the cluster modular group $\bG$.

\subsection{Identification of tropical spaces}

 Let $\widebar{\cX}_{\bR^t}$ denote the tropical cluster space associated to the opposite exchange graph $\widebar{\bE}$. There are PL homeomorphisms $\bar{\phi}_{\bR^t}(s)\colon \widebar{\cX}_{\bR^t}\to \bR^{I(s)}$ and cones $\bar{\sigma}=\bar{\sigma}^a$ indexed by the vertices $a\in V$. Given a cone $\bar{\sigma}$ and a pair of vertices $s_1,s_2\in V$, we denote by  
$\bar{\mu}^{\bar{\sigma}}\colon \bR^{I(s_1)}\to \bR^{I(s_2)}$ the associated linear map \eqref{linear2}.

The definition \eqref{oppo} of the field isomorphisms of the opposite exchange graph implies that there is a
canonical homeomorphism $\iota\colon \cX_{\bR^t}\to \widebar{\cX}_{\bR^t}$ satisfying
\begin{equation}
    \label{ffs}
\bar{\phi}_{\bR^t}(s)(\iota(x))=-\phi_{\bR^t}(s)(x),\end{equation}
for any vertex $s\in V$ and any point $x\in \cX_{\bR^t}$. This should not be confused with the map constructed in the next result. 

\begin{lemma}
    \label{later}
If $\bE$ has finite type there is a unique homeomorphism $\partial\colon \cX_{\bR^t}\to \widebar{\cX}_{\bR^t}$ such that for any vertex $s\in V$
\[\bar{\phi}_{\bR^t}(s)\circ \partial \,|_{\sigma^s}={\phi}_{\bR^t}(s)\,|_{\sigma^s}.\]  
\end{lemma}

\begin{proof}
    Uniqueness is clear since the finite type assumption ensures that the tropical fan is complete.  The given map is clearly continuous on each cone, so we must just check that these maps glue. That is, supposing $x\in \sigma^{s_1}\cap\sigma^{s_2}$, and setting $w=\phi_{\bR^t}(s_1)(x)\in \bR_{\geq 0}^{I(s_1)}$, we must show that  $\bar{\phi}_{\bR^t}(s_1,s_2)(w)={\phi}_{\bR^t}(s_1,s_2)(w)$. We can connect $s_1$ and $s_2$ by a finite sequence of edges $s_1=a_1\to \cdots\to a_n=s_2$ such that $x\in \sigma^{a_i}$ for each $i$. Writing $\tau=\sigma^{s_1}\cap\sigma^{s_2}$ this follows from the completeness of the quotient fan $\Sigma/\tau$ as in the proof of Lemma \ref{aff} below.  Thus it is enough to check the claim for a simple mutation $i\colon s_1\to s_2$. Then it follows from the formula \eqref{log_glue4}, since the condition $x\in \sigma^{s_1}\cap\sigma^{s_2}$ implies that $w(i)=0$.
\end{proof}

\subsection{Tropical duality}

We shall need the following non-obvious relation between the tropical fans associated to the exchange graphs $\bE$ and $\widebar{\bE}$. 

\begin{lemma}\label{duality}\begin{itemize}
    \item[(i)]
 For any two vertices $s_1,s_2\in V $ there is an equality
\begin{equation}
    \label{lurgy}
\bar{\mu}^{\bar{\sigma}^{s_1}}(s_1,s_2)=\mu^{\sigma^{s_2}}(s_1,s_2).\end{equation} 
\item[(ii)] For any vertex $s\in V$, the maximal cones of the fan $\bar{\Sigma}(a)$  associated to the opposite exchange graph $\widebar{\bE}$ are given by the subsets
\[\bar{\sigma}^s(a)=\mu^{\sigma^a}(s,a)\left(\bR^{I(s)}_{\geq 0}\right)\subset \bR^{I(a)}.\]
\end{itemize}
\end{lemma}

\begin{proof}For (i) recall from Remark \ref{remproof} that in the notation of \cite[Section II.2]{Nak}  the vectors $g_j=\mu^{\sigma^{s_2}}(s_2,s_1)(e_j^*)$ are the columns of the matrix ${G}_{s_1}^{s_2}$ associated to the exchange graph $\bE$. Thus the vectors $g'_j={\mu}^{{\sigma}^{s_2}}(s_1,s_2)(e_j^*)$ are the columns of the inverse matrix  $({G}_{s_1}^{s_2})^{-1}$.
Similarly, the vectors $\widebar{g}_j=\bar{\mu}^{\bar{\sigma}^{s_1}}(s_1,s_2)(e_j^*)$ are the columns of the matrix $\widebar{G}_{s_2}^{s_1}$ associated to the opposite exchange graph $\widebar{\bE}$. Combining   \cite[Propositions II.2.3,  II.2.7]{Nak} gives an equality of matrices $\widebar{G}_{s_2}^{s_1}=({G}_{s_1}^{s_2})^{-1}$ which implies $g'_j=\widebar{g}_j$ and therefore proves  the relation \eqref{lurgy}.

Part (ii) is now immediate since by definition $\bar{\sigma}^s(a)=\bar{\mu}^{\bar{\sigma}^s}(s,a)(\bR^{I(s)}_{\geq 0})$.
\end{proof}

\begin{remark}
    As explained in Remark \ref{catdual}, from the perspective of CY$_3$ categorification the relation \eqref{lurgy}  results from the operation of dualising quiver representations, which defines a contravariant equivalence between the categories of representations of a pair of opposite quivers.
\end{remark}

\subsection{DT element}
\label{dt} 

Recall the definition of the cluster modular group $\bG$. By definition it acts faithfully on the underlying oriented graph of $\bE$. As explained in Section \ref{clmogr} it  also acts by homeomorphisms of the space $\cX_{\bR^t}$. These two actions are related by the diagram \eqref{diagramm} which implies that for any element $g\in \bG$ and any vertex $s\in V$ we have
$g(\sigma^s)=\sigma^{g(s)}$.

\begin{definition} We call an element $T\in \bG$ a \emph{DT element} if for all vertices $s\in V$ we have
\begin{equation}\label{tired}\phi_{\bR^t}(T(s),s)(w)(i)=-w(T(i)),\end{equation}
where $w\in \bR_{\geq 0}^{I(T(s))}$ and $i\in I(s)$.\end{definition}

Note that this implies that $\sigma^{T(s)}(s)=\bR_{\leq 0}^{I(s)}$ for all $s\in V$.  
The name DT element  comes from \cite[Definition 3.5]{GonchDT}; see also \cite[Theorem 6.5]{Kel2}.

\begin{prop}
\begin{itemize}
    \item[(i)] There is at most one DT element $T\in \bG$.
    \item[(ii)] If $T\in \bG$ is a DT element then $T$ is central. 
    \item[(iii)] If $\bE$ is of finite type then  there is a DT element $T\in \bG$.
\end{itemize}
\end{prop}

\begin{proof}
For (i), suppose we are given two DT elements $T_i\in \bG$.  Then the composite $g=T_1^{-1}\circ T_2$ satisfies  $\phi_{\bR^t}(s,g(s))(w)=w$ for all vertices $a\in V$ and all $w\in \bR_{\geq 0}^{I(s)}$. Thus $\sigma^{g(s)}=g(\sigma^s)=\sigma^s$ and hence $g(s)=s$ for all vertices $s\in V$. Moreover $g$ also fixes all edges $i\in I(s)$ since these correspond to the facets of $\sigma^s$.  Since $\bG$  acts  faithfully on $\bE$  we conclude that $g=1$ and $T_1=T_2$.

For (ii) let $T\in \bG$ be a DT element and take $g\in \bG$. Then for any vertex $s\in V$  \[g_s\left(\phi_{\bR^t}\left(s,(g^{-1}Tg)(s)\right)(w)\right)=\phi_{\bR^t}\left(g(s),Tg(s))(g_s(w)\right)=-g_s(w),\]
for all $w\in \bR^{I(s)}$, where the linear map $g_s\colon \bR^{I(s)}\to \bR^{I(g(s))}$ is induced by $g\colon I(s)\to I(g(s))$. This implies that $g^{-1}Tg$ is also a DT element so by (i) we have $Tg=gT\in \bG$.

 Part (iii) holds because quivers that are orientations of simply-laced Dynkin diagrams admit maximal green sequences, and such a sequence ensures the existence of a DT element. See \cite[Section 2]{KelGreen} and \cite[Prop. 5.17]{Kel}.
\end{proof}


\section{Quotients of the tropical fan}
\label{secsix}
 From now on   we  assume that our exchange graph $\bE$  is of finite type.  In this section we  introduce  the quotient  of the tropical fan $\Sigma$ by a given tropical cone $\tau$. This material will be used extensively in the next section to understand the structure of the cluster stability space.

\subsection{Partial linearity}

Given an arbitrary cone $\tau\in \Sigma$, we can generalise the construction of Section \ref{linmaps} and define
\begin{equation}
\label{linear}\mu^{\tau}(s_1,s_2)\colon \<\tau(s_1)\>\to \<\tau(s_2)\>\end{equation} 
to be the unique linear map extending $\phi_{\bR^t}(s_1,s_2)\colon \tau(s_1)\to\tau(s_2)$. 
In the case of a simple mutation $i\colon s_1\to s_2\in I $ it is immediate from the formula \eqref{log_glue4} that the PL map $\phi_{\bR^t}(s_1,s_2)$ is linear when restricted to the hyperplane $w(i)=0$. The next lemma gives a generalisation of this  fact.

\begin{lemma}
\label{aff}
Take  $s_1,s_2\in V $ and a cone  $\tau\subseteq \sigma^{s_1}\cap \sigma^{s_2}$. Then 
\[\phi_{\bR^t}(s_1,s_2)(x+v)=\phi_{\bR^t}(s_1,s_2)(x)+\mu^{\tau}(s_1,s_2)(v),\]
for $ x\in \bR^{I(s_1)}$ and $ v\in \<\tau(s_1)\>\subseteq \bR^{I(s_1)}$. \end{lemma}

\begin{proof}
Let $s\in V $ be a vertex such that $\tau\subseteq \sigma^s$. The images under the quotient map
\[q_\tau(s)\colon \bR^{I(s)}\to \bR^{I(s)}/\<\tau(s)\>\] of the cones $\sigma(s)\in\Sigma(s)$ that contain $\tau(s)$ 
form a fan $\Sigma(s)/\tau(s)$  known as the {quotient fan} \cite[Section 3.1]{Fulton}. Since we are assuming that $\bE$ is of finite type, Theorem \ref{fan}(iii) ensures that $\Sigma(s)$ is complete, and the same is then true of the quotient fan $\Sigma(s)/\tau(s)$. Thus we can take a path which connects the interiors of the cones $\sigma^{s_i}(s)/\tau(s)$ and  passes only through faces of codimension 0 and 1. This corresponds to a finite sequence of edges $s_1=a_1\to \cdots\to a_n=s_2$  such that $\tau\subseteq \sigma^{a_i}$ for each $i$.  Thus it suffices to check the result for a simple mutation, and in this case it follows immediately from  \eqref{log_glue4}. \end{proof}

Let us fix a cone $\tau\in \Sigma$. For every pair of vertices $s_1,s_2\in V $ with $\tau\subseteq \sigma^{s_i}\cap \sigma^{s_2}$  the previous lemma gives rise to a commutative diagram
\begin{equation}
\begin{gathered}\label{smallis2}
\xymatrix@C=1.8em{   \<\tau(s_1)\>\ar[d] \ar[rr]^{\mu^{\tau}(s_1,s_2)} &&\<\tau(s_2)\> \ar[d]\\
\bR^{I(s_1)} \ar[rr]^{\phi_{\bR^t}(s_1,s_2)} \ar_{q_\tau(s_1)}[d]  &&  \bR^{I(s_2)}  \ar^{q_\tau(s_2)}[d]  \\
\bR^{I(s_1)}/\<\tau(s_1)\>  \ar[rr]^{\phi_{\bR^t}(s_1,s_2)} && \bR^{I(s_2)}/\<\tau(s_2)\>   }\end{gathered}
 \end{equation}
where $q_\tau(s_i)$ are the quotient maps, and the induced map in the bottom row is PL.

Let us glue the quotients $\bR^{I(s_i)}/\<\tau(s_i)\>$ for all vertices $s_i\in V $ satisfying $\tau\subseteq \sigma^{s_i}$ using the maps $\phi_{\bR^t}(s_1,s_2)$ in the bottom row of \eqref{smallis2}. This gives a PL space $\cX_{\bR^t}/\<\tau\>$ together with homeomorphisms $\phi_{\bR^t}(s)\colon \cX_{\bR^t}/\<\tau\>\to \bR^{I(s)}/\<\tau(s)\>$ for each vertex $s\in V $ satisfying   $\tau\subseteq \sigma^{s}$. 
The commutativity of the bottom square in \eqref{smallis2} gives rise to a map $q\colon\cX_{\bR^t}\to \cX_{\bR^t}/\<\tau\>$, and there is a commutative diagram
\begin{equation}
\begin{gathered}\label{smallish6}
\xymatrix@C=1.8em{     \cX_{\bR^t} \ar[rr]^{\phi_{\bR^t}(s)} \ar[d]_{q_{\tau}}  && \bR^{I(s)} \ar[d]\ar[d]^{q_{\tau}(s)}\\
\cX_{\bR^t}/\<\tau\>\ar[rr]^{\phi_{\bR^t}(s)}  &&\bR^{I(s)}/\<\tau(s)\>}\end{gathered}
 \end{equation}
for every vertex $s\in V $ satisfying  $\tau\subseteq \sigma^{s}$ .

\subsection{Quotient fans}
Let $\tau\in \Sigma$ be a cone and take a vertex $s\in V $ such that $\tau\subseteq \sigma^s$. As in the proof of Lemma \ref{aff}, the images under the quotient map
$q_\tau(s)\colon \bR^{I(s)}\to \bR^{I(s)}/\<\tau(s)\>$ of the cones $\sigma(s)\in\Sigma(s)$ that contain $\tau(s)$ 
form a quotient fan $\Sigma(s)/\tau(s)$. Suppose we are given two vertices $s_1,s_2\in V $ such that $\tau\subseteq \sigma^{s_1}\cap\sigma^{s_2}$. Theorem \ref{coherence}\,(ii) shows that the map $\phi_{\bR^t}(s_1,s_2)$ in the middle row of the diagram \eqref{smallis2} maps cones to cones and induces a bijection $\Sigma(s_1)\to \Sigma(s_2)$. The corresponding map  in the bottom row therefore also  maps cones to cones and induces a bijection $\Sigma(s_1)/\tau(s_1)\to \Sigma(s_2)/\tau(s_2)$. Thus there is a well-defined set of cones  $\Sigma/\tau$ in the space $\cX_{\bR^t}/\<\tau\>$ which are are in bijection with the cones $\tau'\in \Sigma$ such that $\tau\subseteq \tau'$.

We can view $\Sigma/\tau$ as a fan in the PL space $\cX_{\bR^t}/\<\tau\>$  exactly as in Remark \ref{notquitefan}. The following result shows that the map  $q_\tau\colon \cX_{\bR^t}\to \cX_{\bR^t}/\<\tau\>$ is a map of fans.

\begin{lemma}
\label{fan-map}
Fix a cone  $\tau\in \Sigma$. Then for any  cone $\tau'\in \Sigma$ there is a cone $\sigma\in \Sigma$ with  $\tau\subseteq \sigma$ such that $q_\tau(\tau')\subseteq q_\tau(\sigma)$.  
\end{lemma}

\begin{proof}
It suffices to consider the case when $\tau'$ is a maximal cone. 
Recall that the fan  $\Sigma/\tau$ is complete and has maximal cones of the form $\sigma/\tau$ for cones $\sigma\in \Sigma$ containing $\tau$. Thus we can find a cone $\sigma\in \Sigma$ with  $\tau\subseteq \sigma$ such that the interior of $q_\tau(\tau')$ meets the cone $q_\tau(\sigma)$. 
Let us write $\sigma=\sigma^s$ for some vertex $s\in V $ and view everything in the chart $\phi_{\bR^t}(s)$.  Then $\sigma(s)=\bR_{\geq 0}^{I(s)}$ is the positive orthant, and since $\tau\subseteq \sigma$ the map $q_\tau(s)\colon \bR^{I(s)}\to \bR^{I(s)}/\<\tau(s)\>$ is the projection onto some subset of the co-ordinates. Thus $q_\tau(s)^{-1}(q_\tau(s)(\sigma(s)))$ is a union of orthants in $\bR^{I(s)}$. By assumption, the interior of the maximal cone $\tau'(s)$ meets this subset. But by Theorem~\ref{coherence}\,(iii), every maximal cone  whose interior meets an orthant is already contained in that orthant. Thus  $q_\tau(s)(\tau'(s))\subseteq q_\tau(s)(\sigma(s))$ as claimed.  
\end{proof}

We conclude with one further consequence of Lemma \ref{aff}.
 
\begin{lemma}
\label{compare}
Take  vertices $s_1,s_2\in V $ and a cone $\tau\subseteq \sigma^{s_1}\cap \sigma^{s_2}$. Suppose we are given maximal cones $\sigma,\sigma'\in\Sigma$ such that the subsets $q_{\tau}(\sigma),q_{\tau}(\sigma')\subseteq \cX_{\bR^t}/\<\tau\>$ are contained in the same maximal cone of the quotient fan $\Sigma/\tau$. Then $\mu^{\sigma}(s_1,s_2)=\mu^{\sigma'}(s_1,s_2)$.
\end{lemma}

\begin{proof}
By the definition of the quotient fan there is a maximal cone  $\sigma''\in\Sigma$ with $\tau\subseteq \sigma''$  such that $q_\tau(\sigma'')$  contains  both
$q_{\tau}(\sigma)$ and $q_{\tau}(\sigma')$. It then suffices to prove the result in the case $\sigma''=\sigma$. Thus we can assume  that $\tau\subseteq\sigma$ and $q_\tau(\sigma')\subseteq q_\tau(\sigma)$. Take $x'\in \sigma'(s_1)$. Then we can find $x\in \sigma(s_1)$ such that $x'-x\in \<\tau(s_1)\>$. We claim that
\begin{equation*}\begin{split}\mu^{\sigma'}(s_1,s_2)(x')&=\phi_{\bR^t} (s_1,s_2)(x')= \phi_{\bR^t} (s_1,s_2)(x)+\mu^{\tau}(s_1,s_2)(x'-x)\\
&=\mu^{\sigma}(s_1,s_2)(x)+\mu^{\sigma}(s_1,s_2)(x'-x)=\mu^{\sigma}(s_1,s_2)(x').\end{split}\end{equation*}
Indeed, the first equality is the definition, the second  is Lemma~\ref{aff} and the third  is the fact, obvious from the definition of the maps $\mu$, that since $\tau\subseteq \sigma$, the linear map $\mu^{\sigma}(s_1,s_2)$ agrees with $\mu^\tau(s_1,s_2)$ on the subspace $\<\tau(s_1)\>\subseteq \<\sigma(s_1)\>$.
\end{proof}
 

\section{The cluster stability space}
\label{secseven}
Let $\bE$ be an exchange graph of finite type.  In this section we introduce the associated cluster stability space $\cS$. We show that $\cS$ is a complex manifold with an integral affine structure, a Poisson structure, and an action of the additive group $\bC$. In Section \ref{cy3} we will identify $\cS$ with a natural quotient of the  space of stability conditions on a CY$_3$ triangulated category. 

\subsection{Cluster stability space}

We begin by defining a PL manifold $\widebar{\cS}:=\cX_{\bR^t}\times \cX_{\bR^t}$. For each vertex $a\in V $ there is a  PL homeomorphism
\begin{equation}
\label{chart}\phi_{\cS}(a)\colon \widebar{\cS}\to \bC^{I(a)}, \qquad \phi_{\cS}(a)(x,y)=\phi_{\bR^t}(a)(x) + i \phi_{\bR^t}(a)(y).\end{equation}
The corresponding transition functions
$\phi_{\cS}(a,b)=\phi_{\cS}(b)\circ\phi_{\cS}(a)^{-1}$
satisfy 
\[\phi_{\cS}(a,b)(u+iv)=\phi_{\bR^t}(a,b)(u)+i\phi_{\bR^t}(a,b)(v).\]

Recall the semi-closed and closed upper half-planes
\[\Hhalf=\{z\in \bC: 0\leq \arg(z)< \pi\}, \qquad \H= \{z\in \bC:\Im(z)\geq 0\}.\]
For each vertex $a\in V $ we define the subset  $\sigma_{\cS}^a= \cX_{\bR^t}\times\sigma^a\subset \widebar{\cS}$. Note that a point $(x,y)\in \widebar{\cS}$ lies in this subset precisely if $\phi_{\cS}(a)(x,y)(i)\in \H$ for all $i\in I(a)$.

\begin{definition}The \emph{cluster stability space} is the subset $\cS\subset \widebar{\cS}$ consisting of those points $(x,y)\in \widebar{\cS}$ such that there exists a vertex $a\in V$   with $\phi_{\cS}(a)(x,y)(i)\in \;\Hhalf$ for  all $i\in I(a)$. \end{definition}

Given a cone $\tau\in \Sigma$, we define  the \emph{relative interior}  and the \emph{open star} \[\inte(\tau)=\tau\,\setminus\, \bigcup_{\tau'\subsetneq \,\tau} \tau' , \qquad \st(\tau)=\bigcup_{\tau\subseteq \tau'}\,\inte(\tau').\]
Note that, since the fan $\Sigma$ is complete, the open star is indeed an open subset of $\cX_{\bR^t}$. Note also that for every point  $x\in \cX_{\bR^t}$ there is a unique cone $\tau\in \Sigma$ such that $x\in \inte(\tau)$.

For each pair of cones $\tau\subseteq \sigma$ in $\Sigma$ with  $\sigma$  maximal  we introduce the subsets of $\widebar{\cS}$
$$
\begin{aligned}
F(\sigma,\tau)&=\big\{(x,y)\in \widebar{\cS}: q_\tau(x)\in\inte(q_\tau(\sigma))\text{ and }y\in\inte(\tau)\big\},\\[-2mm]
U(\sigma,\tau)&=\big\{(x,y)\in \widebar{\cS}: q_\tau(x)\in\inte(q_\tau(\sigma))\text{ and  }y\in\st(\tau)\big\}.
\end{aligned}
$$
It is clear that $F(\sigma,\tau)\subset U(\sigma,\tau)$ and that $U(\sigma,\tau)\subset \widebar{\cS}$ is open.

\begin{lemma} 
\label{lem-F-and-U}\begin{enumerate}
    \item[(i)] The subset $\cS\subset\widebar{\cS}$ is the disjoint union of the subsets $F(\sigma,\tau)$.
    \item[(ii)] The subset $\cS\subset\widebar{\cS}$ is the union of the subsets $U(\sigma,\tau)$, and is therefore open.
\end{enumerate}
\end{lemma}

\begin{proof}
 Note first that the subsets $F(\sigma,\tau)\subset \widebar{\cS}$  are disjoint. Indeed if $(x,y)\in F(\sigma_1,\tau_1)\cap F(\sigma_2,\tau_2)$ then $y\in \inte(\tau_i)$ for $i=1,2$ implies  that $ \tau_1=\tau_2$. Setting $\tau=\tau_1=\tau_2$ we have $\tau\subseteq \sigma_i$ for $i=1,2$. Then $q_{\tau}(x)\in \inte q_{\tau}(\sigma_i)$ for $i=1,2$ implies that   $\sigma_1=\sigma_2$. 
 
 To prove (i) we must show that  $\cS=\bigcup_{(\sigma,\tau)} F(\sigma,\tau)$. Suppose first that $(x,y)\in \cS$. Then there is a vertex $a\in V $ such that $\phi_{\cS}(a)(x,y)(j)\in \;\Hhalf$ for all $j\in I(a)$. Setting  $\sigma=\sigma^a$ we have $y\in \sigma$ and we define $\tau\subseteq \sigma$ to be the unique face   such that $y\in \inte(\tau)$.   Write $u+iv=\phi_{\cS}(a)(x,y)\in \bC^{I(a)}$. Then by definition of $\;\Hhalf$, for every $j\in I(a)$ we either have $v(j)>0$, or $v(j)=0$ and $u(j)>0$. Moreover $q_\tau(a)$ is the projection onto the components of $\bR^{I(a)}$ for which $v(j)=0$.   It follows that $q_{\tau(a)}(u)\in \inte(q_{\tau(a)}(\sigma(a)))$ and hence
 $q_\tau(x)\in \inte(q_\tau(\sigma))$. This proves that $(x,y)\in F(\sigma,\tau)$. Conversely, if $(x,y)\in F(\sigma,\tau)$ we can write $\sigma=\sigma^a$ for some vertex $a\in V $,  and  the same argument in reverse shows that $\phi_{\cS}(a)(x,y)(j) \in \;\Hhalf$ for all $j\in I(a)$, and hence that $(x,y)\in \cS$.
 
To prove (ii)  suppose that $(x,y)\in U(\sigma,\tau)$. Then $y\in \inte(\tau')$ for a unique cone $\tau'\in \Sigma$, and  $y\in\st(\tau)$ then implies that $\tau\subseteq \tau'$. By Lemma~\ref{fan-map} there is a maximal cone $\sigma'\in \Sigma$ such that $\tau'\subseteq \sigma'$ and $q_{\tau'}(\sigma)\subseteq q_{\tau'}(\sigma')$. Now $q_\tau(x)\in q_{\tau}(\sigma)$ implies that $q_{\tau'}(x)\in q_{\tau'}(\sigma)\subseteq q_{\tau'}(\sigma')$. Since $q_\tau(x)\in \inte (q_\tau(\sigma))$, which is an open subset of $\cX_{\bR^t}/\<\tau\>$, the same argument applies to small perturbations of $x$, and so $q_{\tau'}(x)\in \inte (q_{\tau'}(\sigma'))$. This proves that  $(x,y)\in F(\sigma',\tau')\subset \cS$.
  \end{proof}

\subsection{Collection of charts}
\label{charts}

We will now introduce a collection of charts 
$\varpi(\sigma,\tau)$ on the space $\cS$ which will be used to give it the structure of a complex manifold. 

\begin{lemma}
Fix a pair of cones $\tau\subseteq \sigma$ in $\Sigma$ with  $\sigma=\sigma^s$  maximal.
Then  
there is a unique continuous map $\varpi(\sigma,\tau)\colon U(\sigma,\tau)\to \bC^{I(s)}$ such that any vertex $a\in V $ with $\tau\subseteq \sigma^a$ 
\begin{equation}
\label{ph}\varpi(\sigma,\tau)|_{U(\sigma,\tau)\cap \sigma_{\cS}^a}=\mu^{\sigma}(a,s)\circ \phi_{\cS}(a)|_{U(\sigma,\tau)\cap \sigma_{\cS}^a}.\end{equation}
\end{lemma}

\begin{proof}
Uniqueness is clear since by definition $U(\sigma,\tau)$ is contained in the union of the subsets $\sigma_{\cS}^a$ for vertices $a\in V $ such that $\tau\subseteq \sigma^a$. For each such vertex $a\in V $ the map defined by the right-hand side of  \eqref{ph}
 is clearly continuous on  $U(\sigma,\tau)\cap \sigma_{\cS}^a$. So what remains to check is that the maps corresponding to different vertices  coincide on the intersections of their domains.
 
 Let us then take another vertex $b\in V $ with $\tau\subseteq \sigma^b$ and  suppose that $(x,y)\in U(\sigma,\tau)$ lies in the intersection  $\sigma_{\cS}^a\cap \sigma_{\cS}^b$. What we must show is that, when evaluated at $(x,y)$, the map
\[\mu^{\sigma}(b,s)\circ \phi_{\cS}(b)=\mu^{\sigma}(a,s)\circ \mu^{\sigma}(b,a) \circ \phi_{\cS}(a,b)\circ \phi_{\cS}(a)\]
coincides with the right-hand side of \eqref{ph}. Taking real and imaginary parts, this is the statement that the maps $\phi_{\bR^t}(a,b)$ and $\mu^{\sigma}(a,b)$ coincide when evaluated at each of the points $u=\phi_{\bR^t}(a)(x)\in \bR^{I(a)}$ and $v=\phi_{\bR^t}(a)(y)\in \bR^{I(a)}$.

Beginning with $u$, take a maximal cone $\sigma'\in \Sigma$ such that  $x\in \sigma'$. Since $(x,y)\in U(\sigma,\tau)$, we have $q_\tau(x)\in \inte (q_\tau(\sigma))$, so the subset $q_\tau(\sigma')$  meets the interior of the maximal cone $q_{\tau}(\sigma)$. But by Lemma \ref{fan-map},  this subset $q_\tau(\sigma')$ is contained in some maximal cone of the quotient fan $\cX_{\bR^t}/\<\tau\>$. It follows that $q_\tau(\sigma')\subseteq q_\tau(\sigma)$. Since $\tau\subseteq \sigma^a\cap\sigma^b$, Lemma \ref{compare} then implies that $\mu^{\sigma}(a,b)=\mu^{\sigma'}(a,b)$.
The claim follows since by definition $\mu^{\sigma'}(a,b)$ coincides with $\phi_{\bR^t}(a,b)$  on the cone $\sigma'(a)$ containing $u$.  

Moving on to $v$, note  that the assumption $(x,y)\in \sigma_{\cS}^a\cap \sigma_{\cS}^b$ implies that $y\in \tau':=\sigma^a\cap \sigma^b$. 
By Lemma~\ref{fan-map}, there is a maximal cone $\sigma'\in \Sigma$ such that  $\tau'\subseteq \sigma'$ and
$q_{\tau'}(\sigma)\subseteq q_{\tau'}(\sigma')$. Lemma~\ref{compare} then implies that
$\mu^{\sigma}(a,b)=\mu^{\sigma'}(a,b)$.  Since $y\in \tau'\subseteq \sigma'$ the claim again follows since  $\mu^{\sigma'}(a,b)$ coincides with $\phi_{\bR^t}(a,b)$ on the cone $\sigma'(a)$ containing $v$. 
\end{proof}

\begin{lemma}
\label{sq}
    Fix a pair of cones $\tau\subseteq \sigma$ in $\Sigma$ with  $\sigma=\sigma^s$  maximal. Then
    \begin{itemize}
        \item[(i)] when restricted to the subset $F(\sigma,\tau)\subset U(\sigma,\tau)$ the map $\varpi(\sigma,\tau)$ coincides with  $\phi_{\cS}(s)$,
        \item[(ii)] the map $\varpi(\sigma,\tau)$ is a homeomorphism onto its image.
    \end{itemize}
\end{lemma}

\begin{proof}
    Part (i) is immediate by taking $a=s$ in the relation \eqref{ph}. For part (ii) it will be enough to prove that $\varpi(\sigma,\tau)$ is injective, since the claim then follows by invariance of domain. 

Consider a point $(x,y)\in U(\sigma,\tau)$. We can find  a vertex $a\in V$ such that $\tau\subseteq \sigma^a$ and $y\in \sigma^a$. We set $v=\phi_{\bR^t}(a)(y)\in \bR^{I(a)}_{\geq 0}$. The imaginary part of $\varpi(\sigma,\tau)(x,y)$ is then $w=\mu^{\sigma}(a,s)(v)$. Using Lemma \ref{duality}\,(i) we can rewrite this as  $w=\bar{\mu}^{\bar{\sigma}^a}(a,s)(v)$. By definition of $\mu$ this is also $w=\bar{\phi}_{\bR^t}(a,s)(v)$. But Lemma \ref{later} implies that $v=\bar{\phi}_{\bR^t}(a)(\partial(y))$, so we conclude that  $w=\bar{\phi}_{\bR^t}(s)(\partial(y))$.

Consider  two points $(x_i,y_i)\in U(\sigma,\tau)$ with the same image under the map $\varpi(\sigma,\tau)$. Then, since $\partial$ is a homeomorphism, we have $y_1=y_2$, and so we can take $a\in V$ such that both $y_1,y_2\in \sigma^a$.   But since the right-hand side of \eqref{ph} is clearly injective on $U(\sigma,\tau)\cap \sigma_{\cS}^a$ this implies that $x_1=x_2$ also.
\end{proof}

\subsection{Complex structure}

We will now compute the transition functions for the charts $\varpi(\sigma,\tau)$ introduced in the last section. Let us consider two vertices $s_1,s_2\in V$ and write $\sigma_i=\sigma^{s_i}$. Take also subcones $\tau_i\subseteq \sigma_i$. 

\begin{lemma}
\label{pkh}
Suppose  $F(\sigma_1,\tau_1)\cap U(\sigma_2,\tau_2)\neq \emptyset$. Then on $U(\sigma_1,\tau_1)\cap U(\sigma_2,\tau_2)$
\[\varpi(\sigma_2,\tau_2)=\mu^{\sigma_2}(s_1,s_2)\circ \varpi(\sigma_1,\tau_1).\]
\end{lemma}

\begin{proof}
Take $(u,v)\in F(\sigma_1,\tau_1)\cap U(\sigma_2,\tau_2)$. Then $v\in \inte(\tau_1)$ and $v\in \st(\tau_2)$ which implies $\tau_2\subseteq \tau_1$. By Lemma \ref{fan-map} we can find a vertex $a\in V$ such that $\tau_1\subseteq \sigma^a$ and $q_{\tau_1}(\sigma_2)\subseteq q_{\tau_1}(\sigma^a)$. Then $q_{\tau_2}(u)\in  q_{\tau_2}(\sigma_2)$ implies $q_{\tau_1}(u)\in q_{\tau_1}(\sigma_2)\subseteq q_{\tau_1}(\sigma^a)$.  Since also $q_{\tau_1}(u)\in \inte q_{\tau_1}(\sigma_1)$ we conclude that $\sigma^a=\sigma_1$ and hence $q_{\tau_1}(\sigma_2)\subseteq q_{\tau_1}(\sigma_1)$.

Now take a point $(x,y)\in U(\sigma_1,\tau_1)\cap U(\sigma_2,\tau_2)$. Let $\tau\in \Sigma$ be the unique cone such that $y\in \inte(\tau)$, and take a vertex $s\in V$ such that $\tau\subseteq \sigma^s$.  Then $y\in \st(\tau_1)$ implies that $\tau_1\subseteq \sigma^s$, and similarly $\tau_2\subseteq \sigma^s$. When evaluated at $(x,y)$ we then have
\[\varpi(\sigma_1,\tau_1)=\mu^{\sigma_1}(s,s_1)\circ \phi_{\cS}(s), \qquad \varpi(\sigma_2,\tau_2)=\mu^{\sigma_2}(s,s_2)\circ \phi_{\cS}(s).\]
Now $\tau_1\subseteq \sigma^s\cap\sigma_1$ and $q_{\tau_1}(\sigma_2)\subseteq q_{\tau_1}(\sigma_1)$ so Lemma \ref{compare} implies that $\mu^{\sigma_1}(s,s_1)=\mu^{\sigma_2}(s,s_1)$, and the result follows from the relation \eqref{relly}.
\end{proof}

Suppose $(x,y)\in \cS$ lies in the intersection $U(\sigma_1,\tau_1)\cap U(\sigma_2,\tau_2)$ of the domains of two charts $\varpi(\sigma_i,\tau_i)$. Then $(x,y)\in F(\sigma,\tau)$ for a unique  pair $(\sigma,\tau)$, and Lemma \ref{pkh} shows that 
\[
\varpi(\sigma_2,\tau_2)=\mu^{\sigma_2}(s,s_2)\circ \mu^{\sigma_1}(s_1,s)\circ \varpi(\sigma_1,\tau_1)
\]
on the neighbourhood $U(\sigma_1,\tau_1)\cap U(\sigma_2,\tau_2)\cap U(\sigma,\tau)$ of $(x,y)$.  Thus the charts $\varpi(\sigma_i,\tau_i)$ differ by a linear map in an open neighbourhood of any point of $U(\sigma_1,\tau_1)\cap U(\sigma_2,\tau_2)$. It follows that this linear map is the same for all points in a connected component of this intersection. In particular the transition functions  are holomorphic, and $\cS$ becomes a complex manifold.

Since the transition functions of the atlas are linear and preserve the integral lattices $\bZ^{I(s)}\times \bZ^{I(s)}\subset \bC^{I(s)}$ there is a well-defined  {integral affine structure} on  $\cS$. The fact that the transition functions are linear rather than affine linear gives the extra structure of a {canonical holomorphic vector field} $E$. Finally, Lemma \ref{linear-poisson} implies the existence of   a {holomorphic Poisson structure} $\{-,-\}$. Consider a chart $\varpi(\sigma,\tau)\colon U(\sigma,\tau)\to \bC^{I(s)}$ as above. Then by definition, the components $w(i)$ are a local system of integral affine co-ordinates on $U(\sigma,\tau)$.  In these co-ordinates
\[E=\sum_{i\in I(s)} w(i) \frac{\partial}{\partial w(i)}, \qquad \{w_i(s),w_j(s)\}=2\pi i\,v_{ij}(s).\]

  The integral affine structure, vector field $E$, and Poisson structure $\{-,-\}$   
together define  a \emph{period structure with skew form} on $\cS$ in the sense of \cite[Section 2.4]{BQuad}. This is   the natural structure that exists  on the space of stability conditions of a CY$_3$ triangulated category.

\subsection{Additive group action}
 It is natural to try to integrate the canonical vector field $E$ introduced above to define an action of the multiplicative group  $\bC^*$ on the complex manifold $\cS$. In fact, as the following results show, this can only be done after quotienting $\cS$ by the action of the subgroup of the cluster modular group $\bG$ generated by the DT transformation. Without taking this quotient we  instead obtain an action of the universal cover of $\bC^*$.

\begin{lemma}
\label{euler}
The vector field $E$ generates an action of the additive group  $\bC$ on the space $\cS$.
\end{lemma}

\begin{proof}To construct the action of $\bC$ we have to  show that the flow of the vector field $E$ exists for all time $t\in \bC$. In terms of the components $w(j)$ of a chart $\varpi(\sigma,\tau)\colon U(\sigma,\tau)\to \bC^{I(s)}$ this flow is simply $w(j)\mapsto e^{t}\cdot w(j)$. Decomposing $\bC=\bR\times \bR$  by writing $t=r+i\theta$, the first factor  acts by the rescaling action $w(j)\mapsto e^r \cdot w(j)$. Since this action commutes with all chart transition functions it trivially gives rise to a well-defined action of $\bR$ on $\cS$. Thus it is enough to consider the flows for imaginary times $t=i\theta$ corresponding to the rotation action  $w(j)\mapsto e^{i\theta}\cdot w(j)$.

Given a point $p\in \cS$, we consider the set $S\subset \bR$ of elements $\theta\in \bR_{\geq 0}$ such that this flow exists in an interval $I\subset \bR$ containing $0$ and $\theta$. The set of such $\theta$ is clearly open, and it will suffice to  show that it is  also closed. Consider a sequence $\phi_n\in S$ with $\theta_n\to \theta\in \bR$, and let $p_n\in \cS$ be the point obtained by flowing $p$ until time $i\theta_n$. Passing to a subsequence we can assume that $\theta_n$ is monotonic, and that there is a vertex $s\in V$ such that $p_n\in \sigma^s_{\cS}$ for all $n\geq 1$. Thus $w_n:=\phi_{\cS}(s)(p_n)\in \Hhalf^{I(s)}$. Then since the sequence $w_n$ has a  limit $w\in \H^{I(s)}$, the sequence  $p_n\in \cS$ has a limit  $p\in \widebar{\cS}$. We must show that in fact $p\in \cS$.

There are two cases depending on whether the sequence $\theta_n$ is decreasing or increasing. The first case is easy:  since the points $w_n(j)\in \Hhalf$ are rotating in the clockwise direction, their limit $w(j)\in \H$ lies in the semi-closed upper-half plane $\Hhalf\subset \H$, and so by definition $p\in \cS$. In the second case the points $w_n(j)\in \Hhalf$ are rotating in the anti-clockwise direction.
Let us write $p=x+iy$ with $x,y\in \cX_{\bR^t}$, and similarly  $w=u+iv$ with $u,v\in \bR^{I(s)}$. Then for all $j\in I(s)$ we have either $v(j)>0$, or $v(j)=0$ and $u(j)<0$.

There is a unique face  $\tau(s)\subseteq \sigma(s)=\bR_{\geq 0}^{I(s)}$ such that $v\in\inte(\tau(s))$, and then $q_{\tau}(s)(u)\in \inte q_{\tau}(s)(\bR_{\leq 0}^{I(s)})$. By definition of the DT transformation we have $\sigma^{T(s)}(s)=\bR^{I(s)}_{\leq 0}$ so we can rewrite this as  $q_\tau(x)\in \inte q_\tau(\sigma^{T(s)})$.  By Lemma \ref{fan-map} there is a maximal cone $\sigma'\in \Sigma$ such that $\tau\subseteq\sigma'$ and $q_{\tau}(\sigma^{T(s)})\subseteq q_{\tau}(\sigma')$. Since  all small deformations of $q_\tau(x)$ lie in $q_\tau(\sigma^{T(s)})$ and hence in $q_{\tau}(\sigma')$, we conclude that $q_\tau(x)\in \inte q_\tau(\sigma')$ and hence $p\in F(\sigma',\tau)\subset \cS$.
\end{proof}

 The cluster modular group $\bG$ acts on the cluster stability space $\cS$ in the usual way explained in Section \ref{clmogr}.

\begin{lemma}
     The  action of $\pi i\in \bC$ on the space $\cS$ coincides with  the action of the DT transformation $T\in \bG$.
\end{lemma}

\begin{proof}
 Fix a vertex $s\in V$ and consider a point $p\in \cS$ such that $\phi_{\cS}(s)(p)\in \bR_{>0}^{I(s)}\subset \bC^{I(s)}$. Rotations by $0\leq \theta\leq \pi$ remain in the subset $\sigma^s_\cS$ and it follows that  $\phi_{\cS}(s)(\pi i \cdot p)=-\phi_{\cS}(s)(p)$. On the other hand, for every $i\in I(s)$
 
 \begin{equation*}
     \begin{split}
     \phi_{\cS}(s)(T(p))(i)&=\phi_{\cS}(T(s),s) \left(\phi_{\cS}(T(s))(T(p))\right)(i) \\
     &=-\phi_{\cS}(T(s))(T(p))(T(i))=-\phi_{\cS}(s)(p)(i),
     \end{split}
      \end{equation*}
      where we used \eqref{tired} and \eqref{sample}.
 Thus we find $\pi i\cdot p= T(p)$. Now  the actions of both $\pi i \in \bC$ and $T\in \bG$ preserve the integral affine structure, so they are both given by linear transformations in integral affine co-ordinates.  Since they agree on a subset of the form $\bR_{> 0}^{I(s)}$ it follows that they agree everywhere.\end{proof}


\section{Approach via CY$_3$ triangulated categories}
\label{cy3}

This section is concerned with the categorification of cluster combinatorics via CY$_3$ triangulated categories associated to quivers with potential. We certainly do not attempt a full treatment, which would require a long paper in itself, but instead content ourselves with explaining how some of the objects we have introduced occur naturally  in the categorical setting.    The only logical role of this material  is to allow us to identify our cluster stability space $\cS$ with a quotient of an actual space of stability conditions.    The reader can find further details on CY$_3$ categorification of cluster algebras in the papers of Keller \cite{Kel,Kel2} and Nagao \cite{nag}.

\subsection{Exchange graph via simple tilts}
\label{cat}

Let $\bE$ be an exchange graph.
We would like to associate to each vertex $s\in V $  a $\bC$-linear CY$_3$ triangulated category $\cD(s)$. To do this we first choose a generic potential at one vertex and then propagate using mutation of quivers with potential \cite[Section 7.1]{Kel2}. We then consider \cite[Section 7.2 - 7.4]{Kel2} the bounded derived category of left modules for the corresponding complete CY$_3$ Ginzburg algebra $\Pi_3(Q(s),W(s))$ over the field $\bC$. This category has a canonical finite-length heart $\cA(s)\subset \cD(s)$ whose simple objects up to isomorphism are in bijection with the set $I(s)$.    This gives a natural identification $K_0(\cD(s))=\bZ_{I(s)}$. Given $i\in I(s)$, we denote a corresponding simple object by  $S(i)\in \cA(s)$.

We denote by $\Aut \cD(s)$ the group of  $\bC$-linear triangulated auto-equivalences of the category $\cD(s)$ up to isomorphism. For each edge $i\in I(s)$ the simple object $S(i)$ is spherical and defines a  spherical twist $\Tw_{S(i)}\in \Aut(\cD(s))$.  We denote by $\Br(\cD(s))=\<\Tw_{S(i)}:i\in I(s)\> \subset \Aut \cD(s)$ the subgroup generated by these twists \cite[Section 7.4]{Kel2}. For any object $E\in \cD(s)$ there is a triangle
\begin{equation}\label{tri}\Hom^{\bullet}(S(i),E)\tensor_{\bC} S(i)\lRa{ev} E\lra \Tw_{S(i)}(E).\end{equation}

By construction of the quiver $Q(s)$ and our convention to consider left modules we have $v_{ij}=\chi(S_i,S_j)$. A result  of Keller and Yang \cite[Theorem 7.4]{Kel2} shows that for each edge $i\colon s_1\to s_2$ there are two canonical $\bC$-linear triangulated equivalences $\Phi_\pm(i)\colon \cD(s_2)\to \cD(s_1)$ which satisfy\footnote{Compared with \cite[Theorem 7.4]{Kel2} we have exchanged the labels of $\Phi_\pm(i)$}
\begin{equation}\label{phi}\Phi_\pm(i)(S(\rho_i(j)))=\begin{cases} 
\Tw_{S(i)}^{\mp 1}(S(j)) &\text{ if } i\neq j\text{ and }\pm v_{ij}\geq  0, \\  S(j) &\text{ if } i\neq j\text{ and }\pm v_{ij}< 0, \\
S(i)[\pm 1] &\text{ if }i=j.\end{cases}\end{equation}
   The functors $\Phi_\pm(i)$ satisfy  the relations \begin{equation}
      \label{sister}
  \Phi_{\pm}(\epsilon(i))= \Phi_\mp(i)^{-1}, \qquad \Phi_-(i)=\Tw_{S(i)}\circ \,\Phi_+(i).\end{equation}

 At the level of classes in the Grothendieck group $K_0(\cD)$ the formula \eqref{phi}   gives
\begin{equation}\label{phi2}[\Phi_\pm (i)(S(\rho_i(j)))]=\begin{cases} 
[S(j)]\pm v_{ij}[S(i)]  &\text{ if } i\neq j\text{ and }\pm v_{ij}\geq  0, \\  [S(j)]  &\text{ if } i\neq j,\text{ and }\pm v_{ij}< 0, \\
-[S(i)]  &\text{ if }i=j.\end{cases}\end{equation}
where we used the triangle \eqref{tri}.
Computing \[v_{\rho_i(j),\rho_i(k)}=\chi(\Phi_\pm (S(\rho_i(j))),\Phi_\pm (i)(S(\rho_i(k))))\] then gives  the mutation rule \eqref{form}.
 
Given a pair of hearts $(\cA,\cB)$ in a triangulated category $\cD$, we say \cite[Section 7.6]{Kel2} that $\cA$ is a left tilt of $\cB$, or that $\cB$ is a right tilt of $\cA$,  if the equivalent conditions \[\cB\subset \<\cA,\cA[-1]\>, \qquad \cA\subset \<\cB[1],\cB\>\] 
are satisfied. It follows  from \eqref{phi} that the heart $\Phi_+(i)(\cA(s_2))$ is a left tilt of the heart $\cA(s_2)$.

\subsection{Tilts and tropical signs}

Given vertices $a,b\in V $ we call an equivalence $\Phi\colon \cD(a)\to \cD(b)$ allowable if there is a sequence of edges $a=s_0\lRa{i_1} s_1\to \cdots \lRa{i_d} s_d=b$ and signs $\kappa(j)\in \{\pm 1\}$ such that $\Phi\isom \Phi_{\kappa(1)} (i_1)\circ \cdots\circ \Phi_{\kappa(d)}(i_d)$.
Given a vertex $s\in V$, a heart in $\cD(s)$ is called reachable if it is of the form $\Phi(\cA(a))$, where  $a\in V$ is some vertex, and $\Phi\colon \cD(a)\to \cD(s)$ is an  allowable equivalence. As a subgroup of $\Aut \cD(s)$ the group $\Br(\cD(s))$ acts on the set of hearts in $\cD(s)$, and the second relation of \eqref{sister} shows that this action preserves the subset of reachable hearts. The following statement is a  consequence of the results of \cite[Sections 7.7 - 7.8]{Kel2}.\todo{Maybe add a sentence.}

\begin{thm}
\label{hot}
    Fix a vertex $s\in V$.\nopagebreak
    \begin{itemize}
    \item[(i)] Each orbit of reachable hearts in $\cD(s)$ for the action of $\Br(\cD(s))$ contains a unique representative which is a right tilt of the standard heart $\cA(s)$.
    \item[(ii)]Suppose we are given vertices $a_1,a_2\in V$ and allowable equivalences $\Phi_i\colon \cD(a_i)\to \cD(s)$. Then the corresponding reachable hearts $\Phi_i(\cA(a_i))\subset \cD(s)$ lie in the same orbit for the action of $\Br(\cD(s))$ if and only if $a_1=a_2$.
    \item[(iii)] If an element of $\Br(\cD(s))$ fixes a reachable heart  $\cA\subset \cD(s)$ then it also fixes the simple objects of $\cA$ up to isomorphism and hence acts by the identity on $K_0(\cD(s))$. 
    \end{itemize}
\end{thm}

  For a given vertex $s\in V$, Theorem  \ref{hot}\,(ii) gives a bijection between vertices $a\in V$ and  orbits of reachable hearts in $\cD(s)$ under the action of the group $\Br(\cD(s))$.  Recall that the vertices $a\in V $ are also in bijection with maximal tropical cones $\sigma=\sigma^a\in \Sigma$. Combining these  bijections, we will denote by $O^\sigma(s)$ the $\Br(\cD(s))$ orbit of reachable hearts in $\cD(s)$ corresponding to a given maximal cone $\sigma\in \Sigma$. It consists of those hearts which are the image of the standard heart in $\cD(a)$ under an allowable equivalence $\Phi\colon \cD(a)\to \cD(s)$.

Let us fix a maximal cone $\sigma\in \Sigma$. By Theorem \ref{hot}\,(i)  there is    a unique heart $\cA^\sigma(s)\subset \cD(s)$ in the $\Br(\cD(s))$ orbit $O^\sigma(s)$ which is a right tilt of the standard heart $\cA(s)$.
Given two vertices $s_1,s_2\in V$, we can then find an allowable equivalence \[\Phi^\sigma(s_1,s_2)\colon \cD(s_2) \to \cD(s_1)\] such that $\Phi^\sigma(s_1,s_2)(\cA^\sigma(s_2))=\cA^\sigma(s_1)$. 
We define the linear map $\mu^\sigma(s_1,s_2)^*\colon \bZ_{I(s_2)}\to \bZ_{I(s_1)}$ to be the induced map  on Grothendieck groups, which is well-defined by Theorem \ref{hot}\,(iii).

\begin{remark} The equivalence  $\Phi^\sigma(s_1,s_2)(\cA^\sigma(s_2))$ is only well-defined up to post-composition with auto-equivalences of $\cD(s_1)$ which fix all the simple objects of the canonical heart $\cA(s_1)$.  We will sometimes abuse   notation  below by referring to \emph{the} functor $\Phi^\sigma(s_1,s_2)$ as if it was well-defined, but this is just a notational convenience.
\end{remark}

Fix again a maximal cone $\sigma\in \Sigma$ and write $\sigma=\sigma^a$ for some vertex $a\in V $. Given a vertex $s\in V$, 
we define the cone $\sigma(s)\subset \bR^{I(s)}$ to be the image of the positive orthant $\bR^{I(a)}_{\geq 0}$ under the dual map $\mu^\sigma(a,s)\colon \bR^{I(a)}\to \bR^{I(s)}$. Since by definition $\cA(a)=\Phi^\sigma(a,s)(\cA^\sigma(s))$ it follows that $\sigma(s)$ consists of elements of $\bR^{I(s)}$ which evaluate non-negatively on the dimension vectors of objects of $\cA^\sigma(s)$.
 Because $\cA^\sigma(s)$ is a right tilt of the standard heart $\cA(s)$, for each $i\in I(s)$ we have \cite[Lemma 7.7]{Kel2} either $S(i)\in \cA^\sigma(s)$ or $S(i)[-1]\in \cA^\sigma(s)$. 
We set $\kappa_\sigma(i)=+1$ in the first case and $\kappa_\sigma(i)=-1$ in the second case. Then we have the sign coherence property $\kappa_\sigma(i)\cdot w(i)\geq 0$ for all $w\in \sigma(s)$.

Consider again an edge $i\colon s_1\to s_2$ and write $\kappa=\kappa_\sigma(i)$. We claim that $\Phi_\kappa(i)(\cA^\sigma(s_2))=\cA^\sigma(s_1)$ and is therefore a possible choice for $\Phi^\sigma(s_1,s_2)$.\todo{Shall we make this claim be a lemma?} Using \eqref{phi2} it then follows that  $\mu^\sigma(s_1,s_2)^*$ is given by the formula \eqref{log_glue6}. To prove the claim, set $\cA'=\Phi_{\kappa}(i)^{-1}(\cA^\sigma(s_1))$. We must show that $\cA'$ is a right  tilt of $\cA(s_2)$. Thus we must show that for every $j\in I(s_1) $ the object $\Phi_{\kappa}(i)(S(\rho_i(j)))$ lies in $\<\cA^\sigma(s_1),\cA^\sigma(s_1)[1]\>$. When $j\neq i$ the expression \eqref{phi} shows that this object is a universal extension of $S_i$ by $S_j$ and hence  lies in $\cA(s_1)$. But since $\cA^\sigma(s_1)$ is a right tilt of $\cA(s_1)$ we have $\cA(s_1)\subset \<\cA^\sigma(s_1),\cA^\sigma(s_1)[1]\>$. On the other hand $\Phi_{\kappa}(i)(S(\rho_i(i)))=S(i)[\kappa]$ and by the definition of $\kappa=\kappa_\sigma(i)$ this lies in either $\cA^\sigma(s_1)$ or $\cA^\sigma(s_1)[1]$. 

\begin{remark}
\label{catdual}Note that $\Phi^{\sigma_1}(s_1,s_2)(\cA(s_2))$ is the  unique heart in the  orbit $O^{\sigma_2}(s_1)$
which is a left tilt of $\cA(s_1)$, whereas $\Phi^{\sigma_2}(s_1,s_2)(\cA(s_2))$
is the  unique heart in the orbit $O^{\sigma_2}(s_1)$ which is a right tilt of $\cA(s_1)$. We can repeat everything for the opposite exchange graph $\widebar{\bE}$  and the two stories will be related by the contravariant duality functor which identifies representations of a quiver with representations of its opposite. Since this functor has the effect of exchanging left and right tilts we obtain  the   relation \eqref{lurgy}.
\end{remark}

\subsection{Stability space}

We now assume that the exchange graph $\bE$ is of finite type. Let us also fix a vertex $s_0\in V$ and set $\cD=\cD(s_0)$. We also put $\Br(\cD)=\Br(\cD(s_0))$. We let $\Stab(\cD)$ denote the complex manifold of stability conditions on the category $\cD$.  It was shown in \cite{woolf} that this space is connected (in fact contractible). The results of this paper also imply that the heart of any stability condition on $\cD$ is a reachable heart in the sense defined above. It follows from Theorem \ref{hot}\,(iii) that any element $\Phi\in \Br(\cD)$ which fixes a point of $\Stab(\cD)$ acts trivially on the whole space. Thus the quotient $\Stab(\cD)/\Br(\cD)$ is a complex manifold.

\begin{thm}
\label{stab123}
There is an isomorphism of complex manifolds  $\Stab(\cD)/\Br(\cD)\isom \cS$.
 \end{thm}
 
\begin{proof}
    Take a point $p\in \cS$. We can find a unique pair of tropical cones  $\tau\subset \sigma$ such that $p\in F(\sigma,\tau)$. Write $\sigma=\sigma^{s}$ and $z=\phi_{\cS}(s)(p)\in \bC^{I(s)}$. Note that $z_i\in \Hhalf$ for all $i\in I(s)$. There is therefore a unique  stability condition $(Z,\cP)$ on the category $\cD(s)$ such that $\cP((0,1])=\cA(s)$ and  $Z(S(i))=z(i)$ for all $i\in I(s)$. It is more usual to define stability conditions by fixing their central charge together with the subcategory $\cP((0,1])$ but exactly the same argument applies for the other convention. In either case, the fact that $\cA(s)$ has finite-length ensures that the Harder-Narasimhan property is automatic.
    
    Pulling back $(Z,\cP)$ by an allowable equivalence $ \cD(s_0)\to \cD(s)$ gives a stability condition on $\cD=\cD(s_0)$. By Theorem \ref{hot}\,(ii) this gives a well-defined and injective map
    \[f\colon \cS\to \Stab(\cD)/\Br(\cD).\]
    Moreover,  the results of \cite{woolf} show that the heart of any stability condition on $\cD$ is obtained from the standard heart by a finite chain of simple tilts, and is therefore given by a reachable heart. It follows that all stability conditions arise by the above construction, and hence  $f$ is a bijection. It remains to show that it is a biholomorphism.

    Choosing an allowable equivalence $\Phi\colon \cD(s_0)\to \cD(s)$, we can identify  $\Stab(\cD(s_0))$ with $ \Stab(\cD(s))$, so we can essentially replace $s_0$ with $s$. 
    Consider again the point $p\in \cS$ and the corresponding stability condition $(Z,\cP)$ on  $\cD(s)$.  The central charges $Z(S(i))$ for $i\in I(s)$ are local holomorphic co-ordinates on $\Stab(\cD(s))$. We aim to show that they agree with the components of the map $\varpi(\sigma,\tau)\colon U(\sigma,\tau)\to \bC^{I(s)}$ which gives the complex structure on $\cS$.
    
    By definition of the stability condition $(Z,\cP)$, all simple objects $S(i)$ in $\cA(s)$ are stable with phases in the interval $[0,1)$. For  nearby stability conditions $(Z',\cP')$ on $\cD(s)$ they are therefore also stable with phases in the interval $(-1,+1)$. Thus if $\cB=\cP'([0,1))$ then $S(i)\in \< \cB,\cB[-1]\rangle$. This shows that  $\cA(s)$ is a right tilt of $\cB$. As above, it follows from the results of \cite{woolf} that it is also reachable. 
    We can find an equivalence $\Phi\colon \cD(s)\to \cD(a)$ which maps the heart $\cB\subset \cD(s)$ to the standard heart $\cA(a)\subset \cD(a)$. Then since $\cA(s)$ is a right tilt of $\cB$, its image under $\Phi$ is a right tilt of $\cA(a)$, and hence equal to $\cA^\sigma(a)$. Thus $\Phi$ is a possible choice for $\Phi^\sigma(a,s)$, and so the map induced by $\Phi$ on Grothendieck groups is $\mu^\sigma(a,s)^*$. 
    
    The stability condition $(Z',\cP')$ arises by first applying the above construction to some element $p'\in F(\sigma',\tau')$ to get a stability condition on $\cD(a)$ with heart $\cA(a)$, and central charge given by the components of $w=\phi_{\cS}(a)(p')\in \bC^{I(a)}$. This is then  pulled back to $\cD(s)$ via the equivalence $\Phi$ to get a stability condition with heart $\cB$. The central charge $Z'\in \bC^{I(s)}$ is then equal to $\mu^\sigma(a,s)(w)$  exactly as in \eqref{ph}. 
\end{proof}

\begin{remark}
    In \cite[Section 7]{GonchStab} there is a combinatorial definition of a space $\cU_{\cX}$ which embeds as an open subset of the  quotient $\Stab(\cD)/\Br(\cD)$.  The image consists of stability conditions with reachable hearts having at most one stable object of phase $1$.  This space $\cU_{\cX}$ has the advantage that it is well-defined for an arbitrary exchange graph, but unlike $\cS$ it is not preserved by the natural action of $\bC$ on $\Stab(\cD)/\Br(\cD)$ which rotates the central charge. This $\bC$ action on $\cS$  will be essential for the construction of the twistor space in the sequel to this paper.
\end{remark}


\section{The tangent bundle of $\cX_{\bR_+}$}
\label{tangent}

Let $\bE$ be an exchange graph of finite type. In this section we will study the total space $\cT=T\cX_{\bR_+}$ of the tangent bundle of $\cX_{\bR_+}$.  We construct a natural map $\expt\colon \cT\to \cX_{\bC}$ and show that it is  continuously differentiable and a local homeomorphism. 

\subsection{Tangent space fan}

We write $\cT=T {\cX_{\bR_+}}$ for the total space of the tangent bundle of the smooth manifold $\cX_{\bR_+}$, and denote by
\begin{equation}
    \label{pi}
\pi_{\cT}\colon \cT\to \cX_{\bR_+}\end{equation} the canonical projection.
Then $\cT$ is a smooth manifold with a system of smooth charts
\[\phi_{\cT}(s):=d\phi_{\bR}(s)\colon \cT\to \bC^{I(s)},\]
where we used the canonical identifications
\[T{\bR^{I(s)}}=\bR^{I(s)}\times \bR^{I(s)}=\bC^{I(s)}.\]
The transition functions $\phi_{\cT}(s_1,s_2)=\phi_{\cT}(s_2)\circ \phi_{\cT}(s_1)^{-1}$ are smooth maps $\bC^{I(s_1)}\to \bC^{I(s_2)}$.

Consider an edge $i\colon s_1\to s_2\in I$. Given  $u+iv\in \bC^{I(s_1)}$, let us write  \[u'+iv'=\phi_{\cT}(s_1,s_2)(u+iv)\in \bC^{I(s_2)}.\]
Then differentiating \eqref{log_glue2} gives

\begin{gather}
\label{log_glue10} u'(\rho_i(j))=\begin{cases} u(j)+v_{ij}\cdot \log\big(1+e^{+u(i)}\big) &\text{ if } i\neq j\text{ and }v_{ij}\geq 0, \\
u(j)+v_{ij}\cdot\log \big(1+e^{-u(i)}\big) &\text{ if } i\neq j\text{ and }v_{ij}< 0, \\
-u(i) &\text{ if }i=j,\end{cases} \\
 \label{log_glue11} v'(\rho_i(j))=\begin{cases} v(j)+v_{ij}\cdot v(i)\cdot\left(1+e^{-u(i)}\right)^{-1} &\text{ if } i\neq j\text{ and }v_{ij}\geq 0, \\
v(j)-v_{ij}\cdot v(i)\cdot \left(1+e^{+u(i)}\right)^{-1} &\text{ if } i\neq j\text{ and }v_{ij}< 0, \\
-v(i) &\text{ if }i=j.\end{cases}
\end{gather}

For each vertex $s\in V$ we introduce the subset \begin{equation}
    \label{conest}
\sigma_{\cT}^s=\phi_{\cT}(s)^{-1}\left({\H}^{I(s)}\right)\subset \cT.\end{equation}
Given a point $x\in \cX_{\bR_+}$, we also consider the intersection $\sigma^s_{\cT}(x)= \sigma^s_{\cT}\cap \pi_\cT^{-1}(x)$ with the corresponding fibre $\pi_\cT^{-1}(x)=T_x\,\cX_{\bR_+}$ of the projection \eqref{pi}. 

\begin{prop}
\label{sti}
    Fix a point $x\in \cX_{\bR_+}$.  Then the subsets $\sigma^s_{\cT}(x)\subset T_x\,\cX_{\bR_+}$  are all distinct and are the maximal cones of a complete fan. 
\end{prop}

\begin{proof}
Choose a maximal cone $\sigma=\sigma^a\in \Sigma$ and an element $m\in \bZ_{>0}^{I(a)}$.
    We then obtain a deformation of the gluing maps $\phi_{\bC}(s_1,s_2)_t$ as in  Section \ref{secfive}.  When   $t\in \bR_+$ we can restrict these to maps $\phi_{\bR}(s_1,s_2)_t\colon \bR_+^{I(s_1)}\to \bR_+^{I(s_2)}$. Intertwining  with the componentwise $\log$ maps exactly as in Section \ref{posreal} we arrive at deformations
    \[\phi_{\bR_+}(s_1,s_2)_t\colon \bR^{I(s_1)}\to \bR^{I(s_2)}\]
    which coincide with $\phi_{\bR}(s_1,s_2)$ when $t=1$ and with  $\mu^\sigma(s_1,s_2)$ when $t=0$.
    
  Put $u=\phi_{\bR}(a)(x)\in \bR^{I(a)}$ and consider the isomorphism \[d_x \phi_{\bR}(a)\colon T_x\, \cX_{\bR}\to T_u\, \bR^{I(a)}=\bR^{I(a)}.\] Under this isomorphism the subsets $\sigma^s_{\cT}(x)\subset T_x\,\cX_{\bR}$ are identified with the subsets 
    \begin{equation}
        \label{alig}
    \left(d_{u} \phi_{\bR}(a,s)_t\right)^{-1}\left(\bR_{\geq 0}^{I(s)}\right),\end{equation}
    for $t=1$. Allowing $t\in [0,1]$ to vary we find that at $t=0$ they become
    \[\left(\mu^\sigma(a,s)\right)^{-1}\left(\bR_{\geq 0}^{I(s)}\right)=\mu^\sigma(s,a)\left(\bR_{\geq 0}^{I(s)}\right).\]
    By Lemma \ref{duality}(ii) these are precisely the maximal cones $\bar{\sigma}^s(a)$ of the tropical fan associated to the opposite exchange graph.    

    We now apply the c.c.~cone deformation argument of Appendix~\ref{fan-def-section}. What we must show is that given an edge $i\colon s_1\to s_2$ the cones \eqref{alig} corresponding to the vertices $s_1,s_2$ meet in a common facet. This reduces to the same statement for the cones
    \[\bR_{\geq 0}^{I(s_1)}, \qquad \left(d_{h} \phi_{\bR}(s_1,s_2)_t\right)^{-1}\left(\bR_{\geq 0}^{I(s_2)}\right),\]
    where $h=\phi_{\bR}(a,s_1)_t(u)$.
This follows from the $t$-deformed version of \eqref{log_glue11} obtained by differentiating \eqref{log_glue7}. Indeed, this shows that the intersection of the above two cones is cut out by the equation $v(i)=0$. 
\end{proof}

\subsection{Almost complex structure}

For every vertex $s\in V$ the chart $\phi_\cT(s)$ induces a complex structure $I(s)$ on the smooth manifold $\cT$. The underlying almost complex structure gives an automorphism 
$I_p(s)\colon T_p \,\cT\to T_p\,\cT$ satisfying $I_p(s)^2=-\id$ for every point $p\in \sigma_{\cT}^s\subset \cT$. We now show that we can glue these almost complex structures  to produce a continuous almost complex structure on $\cT$. 

\begin{lemma}
\label{lem-rotate-I}
    There is a continuously varying family of automorphisms $I_p\in \Aut(T_p\,\cT)$ such that $I_p=I_p(s)$ for all $p\in \sigma_{\cT}^s$.
\end{lemma}

\begin{proof}We must show that if $p\in \sigma^{s_1}_{\cT}\cap  \sigma^{s_2}_{\cT}$ then $I_p(s_1)=I_p(s_2)$ as elements of $\Aut(T_p\,\cT)$.
    Connecting $s_1$ and $s_2$ by a chain of edges reduces the assertion to the situation of a single edge $i\colon s_1\to s_2$.
    We need to show that the derivative of the transition function $\phi_{\cT}(s_1,s_2)\colon \bC^{I(s_1)}\to \bC^{I(s_2)}$  at the point $p_1=\phi_{\cT}(s_1)(p)$ is complex linear. Let us write $u_j=u(j)$ and $u'_j=u(\rho_i(j))$ for all $j\in I(s_1)$.  Similarly we set $v_j=v(j)$ and $v'_j=v(\rho_i(j))$. The assumption that $p\in \sigma^{s_1}_{\cT}\cap  \sigma^{s_2}_{\cT}$ gives $v_i=0$. Differentiating the formulae \eqref{log_glue10}-\eqref{log_glue11} we get
    \[d_{p_1}\phi_{\cT}(s_1,s_2)\left(\frac{\partial}{\partial u_j}\right)=\frac{\partial}{\partial u'_j},\qquad d_{p_1}\phi_{\cT}(s_1,s_2)\left(\frac{\partial}{\partial v_j}\right)=\frac{\partial}{\partial v'_j},\]
    when $j\neq i$. We also get
    \[d_{p_1}\phi_{\cT}(s_1,s_2)\left(\frac{\partial}{\partial u_i}\right)=-\frac{\partial}{\partial u'_i}+\sum_{v_{ij}>0} \frac{v_{ij}}{1+e^{-u_i}} \cdot \frac{\partial}{\partial u'_j} -\sum_{v_{ij}<0} \frac{v_{ij}}{1+e^{u_i}} \cdot \frac{\partial}{\partial u'_j},\]
   \[d_{p_1}\phi_{\cT}(s_1,s_2)\left(\frac{\partial}{\partial v_i}\right)=-\frac{\partial}{\partial v'_i}+\sum_{v_{ij}>0} \frac{v_{ij}}{1+e^{-u_i}} \cdot \frac{\partial}{\partial v'_j} -\sum_{v_{ij}<0} \frac{v_{ij}}{1+e^{u_i}} \cdot \frac{\partial}{\partial v'_j}.\]
   There are some extra terms in the first formula which vanish because $v_i=0$. It is then clear that $d_{p_1}\phi_{\cT}(s_1,s_2)$ is complex linear, since multiplication by $i$ maps $\frac{\partial}{\partial u_i}$ to $\frac{\partial}{\partial v_i}$ and similarly for the primed variables.
\end{proof}

The derivative of the projection \eqref{pi} at a point $p\in \cT$ defines a map \[d_p\, \pi_{\cT}\colon T_p\, \cT\to T_{x}\,  \cX_{\bR_+}, \qquad  x=\pi_{\cT}(p).\]
We define the subspace $F_p=\ker(d_p\pi_{\cT})\subset T_p \,\cT$ of vertical tangent vectors. It is canonically identified with the tangent space $T_p(T_x\, \cX_{\bR_+})$ to the fibre $\pi_{\cT}^{-1}(x)=T_x \,\cX_{\bR_+}$.

\begin{lemma}
\label{lem}
    There is a decomposition $T_p\, \cT=H_p\oplus F_p$ where $H_p=I_p(F_p)$.
\end{lemma}

\begin{proof}
    The images of the subspaces $H_p$ and $F_p$ under the derivative of the chart $\phi_{\cT}(s)$ are the subspaces spanned by the vectors $\frac{\partial}{\partial {u_j}}$ and $\frac{\partial}{\partial {v_j}}$ respectively. 
\end{proof}

\subsection{Tangent cones}
Given a point $p\in \cT$, there is a corresponding subset of vertices of the exchange graph
\[V(p)=\{s\in V: p\in \sigma_{\cT}^s\}\subset V. \]
For each element $s\in V(p)$  
there is a a polyhedral cone $T_p \, \sigma_\cT^s\subset T_p\, \cT$ which is the tangent cone to the subset $\sigma_{\cT}^s\subset \cT$ at the point $p$. We will show that these subsets are the maximal cones of a  complete fan in the tangent space $\cT_p\, \cT$.  

Rather than introducing the formalism of tangent cones we  proceed  as follows.
We first consider the space $\bC^{I(s)}$ with co-ordinates $w(j)=u_j+iv_j$ for $j\in I(s)$.  The tangent space to $\bC^{I(s)}$ at any point  is a real vector space spanned by the vectors $\frac{\partial}{\partial {u_j}}$ and $\frac{\partial}{\partial {v_j}}$. We consider a point $w\in \H^{I(s)}\subset \bC^{I(s)}$ so that $v_j\geq 0$ for all $j\in I(s)$. We define the tangent cone to the subset $\H^{I(s)}\subset \bC^{I(s)}$ at the point $w$ to be
\begin{equation}\label{tancone}T_w\,  \H^{I(s)} =\left\{\big.\sum_j \left(a_j \frac{\partial}{\partial {u_j}}+b_j \frac{\partial}{\partial {v_j}}\right):v_j=0\implies b_j\geq 0\right\}\subset T_w \, \bC^{I(s)}.\end{equation}
We can then define the tangent cone to the subset $\sigma_{\cT}^s\subset \cT$ at a point $p\in \sigma_{\cT}^s$  to be
\begin{equation}\label{blah}T_p \, \sigma_\cT^s=\left(d_p\phi_{\cT}(s)\right)^{-1} \left( T_{\phi_{\cT}(s)(p)} \, \H^{I(s)}\right).\end{equation}

\begin{lemma}
\label{compfan}
     Fix a point $p\in \cT$. Then the subsets $T_p \, \sigma_\cT^s\subset T_p\,\cT$  for vertices $s\in V(p)$ are all distinct, and are the  maximal cones of a complete fan. Each of these cones contains the subspace $H_p\subset T_p\, \cT$.
\end{lemma} 

\begin{proof}
The point $p\in \cT$ is an element of $T_x \,\cX_{\bR_+}$ where $x=\pi_{\cT}(p)$. The localisation of the complete fan of Lemma \ref{sti} at this point gives a complete fan in the vector space $T_p(T_x\, \cX_{\bR_+})$ whose maximal cones are indexed by the set $V(p)$. This vector space is canonically identified with the  subspace $F_p\subset T_p\,\cT$  of vertical tangent vectors for the map $\pi_{\cT}\colon \cT\to \cX_{\bR_+}$. Pulling back the localised fan by the projection  $T_p\,\cT\to F_p$ with kernel $H_p$ gives a complete fan in $T_p\, \cT$. We claim that its maximal cones coincide with the subsets  $T_p \, \sigma_\cT^s\subset T_p\, \cT$.

To prove the claim we can take a vertex $s\in V(p)$ and work in the chart $\phi_{\cT}(s)$. The cone $\sigma_\cT^s$ is mapped to the standard cone $\H^{I(s)}\subset \bC^{I(s)}$, and the localisation of this cone at the point $w=\phi_{\cT}(s)(p)$ is precisely the tangent cone \eqref{tancone}. On the other hand, the images of the subspaces $H_p$ and $F_p$ under the derivative of the chart $\phi_{\cT}(s)$ are the subspaces spanned by the vectors $\frac{\partial}{\partial {u_j}}$ and $\frac{\partial}{\partial {v_j}}$ respectively. The claim then follows, and this gives the result.
\end{proof}

\subsection{Exponential map}
The following map will play an essential role in what follows.

\begin{lemma}
\label{lem-cont-exp2}
There is a unique continuous map $\expt\colon \cT\to \cX_{\bC}$ such that for all vertices $s\in V$
\begin{equation}
    \label{defining_relation2}
\phi_{\bC}(s)\circ \expt|_{\sigma_\cT^s}= \exp\circ \, \phi_{\cT}(s)|_{\sigma_\cT^s},\end{equation}
where $\exp$ on the right-hand side denotes  the  componentwise exponential \begin{equation}
\label{compexp2}
\exp\colon \bC^{I(s)}\to(\bC^*)^{I(s)}.\end{equation}
\end{lemma}

\begin{proof}
Take $p\in \sigma^{\cT}_{s_1}\cap \sigma^{\cT}_{s_2}$. We must show that  taking $s=s_1$ or $s=s_2$ in \eqref{defining_relation2} gives the same answer for $\exp_{\cT}(p)$. By the usual argument it  is enough to consider the case of a single edge $i\colon s_1\to s_2$. Write $\phi_{\cT}(s_1)(p)=u+iv$ with $u,v\in \bR^{I(s_1)}$. We must show that
\begin{equation}
\label{navalny2}\exp(\phi_{\bT}(s_1,s_2)(u+iv)) =\phi_{\bC}(s_1,s_2)(\exp(u+iv))\in (\bC^*)^{I(s_2)}.\end{equation}
Since $p\in \sigma^{\cT}_{s_1}\cap \sigma^{\cT}_{s_2}$ we have $v(i)=0$. 
Evaluated on an edge $\rho_i(j)$  both sides of \eqref{navalny2} are therefore given by the expression
\begin{equation}\label{gluvee2}\begin{cases} 
e^{u(j)+iv(j)}\left (1+e^{+u(i)}\right)^{v_{ij}} &\text{ if } i\neq j\text{ and }v_{ij}\geq  0, \\ e^{u(j)+iv(j)}\left (1+e^{-u(i)}\right)^{v_{ij}} &\text{ if } i\neq j\text{ and }v_{ij}<0, \\
e^{-u(i)} &\text{ if }i=j,\end{cases}\end{equation}
which proves the claim. 
\end{proof}

Recall that given a continuous function $f\colon \bR^n\to \bR^m$ the directional derivative of $f$ at a point $x\in \bR^n$ along a tangent vector $u\in T_x \, \bR^n =\bR^n$ is the limit
\[d_x(f)(u):=\lim_{t\to 0^+} \frac{f(x+tu)-f(x)}{t}\in \bR^m, \]
whenever this exists. One can then define the directional derivative of  continuous map between smooth manifolds in the obvious way.  

\begin{lemma}
\label{items}\begin{itemize}
    \item[(i)]Given a point $p\in \cT$, the directional derivative $d_p \expt(u)$ of the map $\expt\colon \cT\to \cX_{\bC}$  along a tangent vector $u\in T_p\, \cT$ always  exists.

    \item[(ii)] These directional derivatives combine to give a continuous  map \[d_p \expt\colon T_p\, \cT\to T_q\, \cX_{\bC}, \qquad  q=\expt(p).\]

 \item[(iii)] For each $s\in V(p)$  the map $d_p \expt$ is linear and injective on the cone $T_p \, \sigma^s_{\cT}\subset T_p\, \cT$.
 
\item[(iv)] If $J_q$ denotes the complex structure on $T_q\, \cX_{\bC}$ then
\begin{equation}
    \label{hh}d_p \expt\circ I_p=J_p\circ d_p \expt.
\end{equation}
\end{itemize}
    \end{lemma}

    \begin{proof}
    Take a point $p\in \cT$ and a tangent vector $u\in T_p\,\cT$. By Lemma \ref{compfan} we can take a vertex $s\in V(p)$ with $u\in T_p \, \sigma^s_{\cT}$. We can then compute the directional derivative using the charts $\phi_{\cT}(s)$ and $\phi_{\bC}(s)$. By the defining relation \eqref{defining_relation2} this amounts to computing the directional derivative of the componentwise exponential \eqref{compexp2} at the point $w=\phi_{\cT}(s)(p)\in \bC^{I(s)}$ in the direction $d_p\phi_{\cT}(s)(u)\in T_w \, \H^{I(s)}$. Note that the definition \eqref{tancone} shows that small perturbations of $w$ in this direction remain in the subset $\H^{I(s)}\subset \bC^{I(s)}$. The claims now all follow easily from the fact that the componentwise exponential is holomorphic and has injective derivative.
    \end{proof}
    
\subsection{Local homeomorphism}
The main result of this section is the following.

\begin{thm}
\label{helge's}
    The map 
    $\expt\colon \cT\to \cX_{\bC}$
    is $C^1$ and  is  a local  homeomorphism. 
\end{thm}

\begin{proof} 
To show that $\expt$ is differentiable we must show that for every $p\in \cT$ the PL map $d_p \expt$
 is in  fact linear. 
Consider the decomposition $T_p \, \cT=F_p\oplus H_p$ of Lemma \ref{lem}. Given $h\in H_p$ and $u\in T_p\, \cT$, we claim that \begin{equation}
     \label{lakes}
 d_p \expt(u+h)=d_p \expt(u)+d_p \expt(h).\end{equation} Indeed, by Lemma \ref{compfan} we can take a vertex $s\in V(p)$ with  $u\in T_p \, \sigma^s_{\cT}$. Moreover  $h\in H_p\subset T_p \, \sigma^s_{\cT}$.  The claim then follows from Lemma \ref{items}\,(iii).
Note  that \eqref{lakes} implies in particular that the map $d_p \expt$ is linear when restricted to the subspace $H_p$. It follows from  Lemma \ref{items}\, (iv) that it is also linear when restricted to  the subspace $F_p=I_p(H_p)$.

To prove that $d_p \expt$ is linear consider $u_1,u_2\in T_p\, \cT$ and write $u_i=f_i+h_i$ with $f_i\in F_p$ and $h_i\in H_p$. Set $u=u_1+u_2$ and $h=h_1+h_2$ and $f=f_1+f_2$. Then, writing $\theta=d_p \expt$ we have

\begin{equation*}
\begin{split}\theta(u_1+u_2)&=\theta(f_1+h_1+f_2+h_2)=\theta(f+h)=\theta(f)+\theta(h)\\ &=\theta(f_1)+\theta(f_2)+\theta(h_1)+\theta(h_2)=\theta(u_1)+\theta(u_2),\end{split}\end{equation*}
where in the first line we applied \eqref{lakes}, and in the second line we used the linearity of $\theta|_{F_p}$ and $\theta|_{H_p}$ and then applied \eqref{lakes} again.

Lemma \ref{items}\,(iii) shows that the map $d_p \expt$ is injective on  each cone $T_p\, \sigma^s_{\cT}$. Since it is also linear  it is an isomorphism. To show that this linear map varies continuously with $p\in \cT$ note that it is enough to check this on a given cone.\todo{Add a sentence.} The result now follows from the inverse function theorem.
\end{proof}


\section{The log cluster space}
\label{secnine}

Let $\bE$ be an exchange graph of finite type. In this section we introduce the log cluster space $\cL$ which will be the general fibre of our cluster twistor space.  We construct a homeomorphism $h\colon \cL\to \cT$ which  allows us to apply Theorem \ref{helge's} to prove that the natural map $ \expl\colon \cL\to \cX_{\bC}$ is a local homeomorphism. This then gives   $\cL$ the structure of a complex manifold.

\subsection{Definition}

We define the log cluster space to be the product
$\cL=\cX_{\bR_+}\times \cX_{\bR^t}$. There is an obvious projection onto the first factor
\begin{equation} \label{pi2}\pi_{\cL}\colon \cL\to \cX_{\bR_+}.\end{equation}
 For each vertex $s\in V$ there is a homeomorphism
\[\phi_{\cL}(s)\colon \cL\to \bC^{I(s)}, \qquad \phi_{\cL}(s)(x,y)=\phi_{\bR}(s)(x) + i \phi_{\bR^t}(s)(y).\]
 The corresponding transition functions 
  $\phi_{\cL}(s_1,s_2)=\phi_{\cL}(s_2)\circ \phi_{\cL}(s_1)^{-1}$ are given by
\begin{equation}
\label{trans}
\phi_{\cL}(s_1,s_2)(u+iv)=\phi_{\bR}(s_1,s_2)(u)+i\phi_{\bR^t}(s_1,s_2)(v).
\end{equation}
For each vertex $s\in V$  we define the subset $\sigma_{\cL}^s=\cX_{\bR_+}\times \sigma^s\subset \cL$. Since the tropical fan $\Sigma$ is complete, these subsets cover $\cL$. 

We can construct a map  $h\colon  \cL\to \cT $ by gluing the natural systems of cones in the two spaces.

\begin{proposition}
\label{prop-TX-L-homeo}
        There is a unique homeomorphism $ h\colon  \cL\to \cT$ such that for any $a\in V$
        \[h|_{\sigma^a_{\cL}}= \phi_\cT(a)^{-1}\circ \phi_{\cL}(a)|_{\sigma^a_{\cL}}.\]
        Moreover this homeomorphism intertwines the  projections \eqref{pi} and \eqref{pi2}.
\end{proposition}

\begin{proof}
    For any pair of vertices $s_1,s_2\in V $ we need to show that the maps $\phi_\cT(s_i)^{-1}\circ \phi_{\cL}(s_i)$ agree on the intersection  $\sigma^{s_1}_\cL\cap \sigma^{s_2}_\cL$. As the general case follows from concatenating edges, it suffices to check this for a single edge $i:s_1\to s_2$. The intersection $\sigma^{s_1}_\cL\cap \sigma^{s_2}_\cL$ then maps into the real hyperplane $\Im w(i)=0$ under the map $\phi_{\cT}(s_1)$. On this hyperplane, the chart transition map 
$\phi_\cL(s_1,s_2)= \phi_\bR(s_1,s_2)+i\phi_{\bR^t}(s_1,s_2)$ agrees with $\phi_{\cT}(s_1,s_2)$ as can be seen by comparing \eqref{log_glue2} and \eqref{log_glue4} with \eqref{log_glue10} and \eqref{log_glue11}. This establishes the continuity of $h$.

The compatibility $\pi_\cL=\pi_\cT\circ h$   is clear by definition.
By invariance of domain, to prove that $h$ is a homeomorphism it is enough to show that it is bijective, and we can check this in each fibre of $\pi_{\cL}$. Since the map preserves the fans and is an isomorphism on each maximal cone we are done.\todo{Add a word or two here.}
\end{proof}

\subsection{Exponential map}
The following result is a formal consequence of Lemma \ref{lem-cont-exp2} and Proposition \ref{prop-TX-L-homeo}, but since the map $ \expl$  is fundamental we include a direct proof.

\begin{lemma}
\label{lem-cont-exp}
There is a unique continuous map $ \expl\colon \cL\to \cX_{\bC}$ such that for all vertices $s\in V$
\begin{equation}
    \label{defining_relation}
\phi_{\bC}(s)\circ  \expl|_{\sigma_\cL^s}= \exp\circ \, \phi_{\cL}(s)|_{\sigma_\cL^s},\end{equation}
where $\exp$ on the right-hand side again denotes  the  componentwise exponential \eqref{compexp2}.
\end{lemma}

\begin{proof}
Take $(x,y)\in \cX_{\bR}\times \cX_{\bR^t}$ and suppose that $y\in \sigma^{s_1}\cap \sigma^{s_2}$.  We must show that  taking $s=s_1$ or $s=s_2$ in \eqref{defining_relation} gives the same answer for $\exp_{\cL}(x,y)$. As usual, it is enough to consider the case of a single edge $i\colon s_1\to s_2$. Write $\phi_{\cL}(s_1)(x,y)=u+iv$ with $u,v\in \bR^{I(s_1)}$. We must show that
\begin{equation}
\label{navalny}\exp(\phi_{\bR}(s_1,s_2)(u) + i\phi_{\bR^t}(s_1,s_2)(v))=\phi_{\bC}(s_1,s_2)(\exp(u+iv))\in (\bC^*)^{I(s_2)}.\end{equation}
Since $y\in \sigma^{s_1}\cap \sigma^{s_2}$ we have $v(i)=0$. 
Evaluated on an edge $\rho_i(j)$  both sides of \eqref{navalny} are therefore given by the expression
\begin{equation}\label{gluvee}\begin{cases} 
e^{u(j)+iv(j)}\left (1+e^{+u(i)}\right)^{v_{ij}} &\text{ if } i\neq j\text{ and }v_{ij}\geq  0, \\ e^{u(j)+iv(j)}\left (1+e^{-u(i)}\right)^{v_{ij}} &\text{ if } i\neq j\text{ and }v_{ij}<0, \\
e^{-u(i)} &\text{ if }i=j,\end{cases}\end{equation}
which proves the claim. 
\end{proof}

It follows  that the space $\cL$ is obtained by gluing the cells $\sigma^\cL_{s}\isom \H^{I(s)}$ along their boundaries using the logarithms of the maps $\phi_{\bC}(s_1,s_2)$. This is how the space was described in the Introduction.

\subsection{Local homeomorphism}
 
Note that the definitions of the various maps imply the relation $\expl=\expt\circ\, h$. 
Theorem \ref{helge's} then immediately gives

\begin{thm}
\label{keytwo}
The map $ \expl\colon \cL\to \cX_{\bC}$ is a local homeomorphism.
\end{thm}

We now equip the space $\cL$ with the unique complex manifold structure for which the  map $ \expl$ is {\'e}tale. For any vertex $s\in V$ the components $w(i)=\phi_{\cL}(s)(i)$ are holomorphic functions on the interior of the subset $\sigma_\cL^s\subset \cL$. Note however that because $\phi_{\bR_+}(s_1,s_2)\neq\phi_{\bR^t}(s_1,s_2)$, these components are not in general holomorphic (or smooth, or even $C^1$) on the interiors of other cones $\sigma^\cL_{s'}$.

Pulling back the Poisson structure on $\cX_{\bC}$ via the map $ \expl$  gives a holomorphic Poisson structure on $\cL$. For any vertex $s\in V$ and edges $j,k\in I(s)$ we have  \[\left\{w(j),w(k)\right\}|_{\sigma_\cL^s}=2\pi i \cdot v_{jk}.\]

As a complex manifold $\cL$ has an underlying smooth structure. Theorem \ref{helge's} gives

\begin{thm}
    The map $h\colon \cL\to \cT$ is a $C^1$ homeomorphism.
\end{thm}

Recall from Section \ref{tropclsp} the subset $\cX_{\bZ^{t}}\subset \cX_{\bR^t}$ of integer points of the tropical cluster space. By Lemma \ref{feb}(iii), multiplication by $r\in \bR_+$  is a well-defined operation on $\cX_{\bR^t}$, so we can also consider the subset $ \cX_{2\pi\bZ^t}:=2\pi \cX_{\bZ^t}\subset \cX_{\bR^t}$.

\begin{lemma}
The fibre of $ \expl$ over a point of $\cX_{\bR_+}\subset \cX_{\bC}$ is naturally identified with $ \cX_{2\pi \bZ^t}$.
\end{lemma}

\begin{proof}
Take a point $p\in \cX_{\bR_+}$ and consider the fibre $ \expl^{-1}(p)$ as a subset of $\cX_{\bR_+}\times \cX_{\bR^t}$. We claim  that it coincides with the subset $\{p\}\times  \cX_{2\pi\bZ^t}$. To prove this it is enough to show that for each vertex $s\in V $ the intersections of the two subsets with the cone $\sigma_\cL^s$ are the same. Write $w=\phi_{\bR_+}(s)(p)\in \bR_+^{I(s)}$.  Consider a point $x+iy\in \sigma_\cL^s$ and set $\phi_{\cL}(s)(x+iy)=u+iv$. Then   the defining relation \eqref{defining_relation} shows that $ \expl(x+iy)=p$ precisely if $w(j)=e^{u(j)+iv(j)}$ for all $j\in I(s)$. This is the case if and only if $w(j)=e^{u(j)}$ and $v(j)\in (2\pi)\cdot  \bZ$.  This is the condition  that $x=p$ and $y\in  \cX_{2\pi\bZ^t}$.
\end{proof}

\begin{remark}
    The map $ \expl$ is not a covering map at points of $\cX_\bC$ that are not contained in all cluster tori. Indeed consider a path $p=p(t)$ in $\cX_{\bC}$ such that in some given chart $\phi_{\bC}(s)$ we have $x_i(t):=\phi_{\bC}(s)(p(t))(i)\to 0$ along a ray as $t\to 0$. Then we can lift this to a path in the corresponding cone $\sigma_\cL^s$ in $\cL$ by writing $x_i(t)=\exp(w_i(t))$ with $\Im w_i(t)$ constant. This path does not have a limit in $\cL$.\todo{Add an extra sentence or two.}
\end{remark}


\section{Geometric descriptions in the A$_n$ case}
\label{secsurface}

 Let us consider the case of the A$_n$ quiver with $n\geq 1$. In the literature one can find descriptions of most of the spaces we have been considering in terms of objects of surface topology. These are special cases of a more general story involving an arbitrary marked bordered surface. Here the relevant surface is the disc with $m=n+3$ marked points on the boundary, or equivalently the $m$-gon. Similar descriptions can be given in the $D_n$ case when the relevant surface is a once-punctured disc with $n$ marked points on the boundary, but we will not include these here.

All results in this section can be found elsewhere in the literature, often  in far more general contexts. In particular,  Sections \ref{121} and \ref{122} contain simple cases of far more general results of Fock and Goncharov \cite{FG0}, Section \ref{123} follows from the work of Bridgeland and Smith \cite{BS}, and Section \ref{124} contains a special case of work of Gupta and Mj \cite{GM}.  There is a substantial overlap with the paper \cite{All} by Allegretti which we recommend as a further reference.

 \subsection{Exchange graph}
 \label{121}

Set $\omega=\exp(2\pi i/m)$ and let $P$ be the $m$-gon viewed as a polygon in $\bC$ with vertices  at the points $\omega^k$ corresponding to elements $k\in \bZ_{m}$.  By a triangulation of $P$ we mean a collection of $n=m-3$ distinct, non-crossing isotopy classes of arcs connecting the vertices of $P$.  To avoid confusion with the edges of the exchange graph we will refer to arcs of a triangulation rather than  edges.
Given a triangulation $T_1$ and an arc $A_1\in T_1$, there is a unique triangulation $T_2\neq T_1$ such that   $T_1\,\setminus\,\{A_1\}=T_2\,\setminus\,\{A_2\}$ for some arc  $A_2\in T_2$. The triangulation $T_2$ is called the flip of $T_1$ in the arc $A_1$. 

We introduce the unoriented graph $\bE$   whose vertices $V$ are the triangulations of $P$, and whose oriented edges $I$ correspond to the operation of flipping an arc. Thus elements of the set $I$ consist of a triangulation together with a chosen arc. In the setting of the previous paragraph, the involution $\epsilon\colon I\to I$ of Section \ref{notconv} sends the flip of  the triangulation $T_1$ in the arc $A_1$ to the flip of the triangulation $T_2$ in the arc $A_2$.
 Given a vertex $s\in V$, we denote by $T(s)$ the corresponding triangulation. Note that  the set $I(s)$ of oriented edges emanating from a vertex $s\in V$ is identified with the set of arcs of the triangulation $T(s)$. We denote by $A(i)$ the arc of $T(s)$ corresponding to an element $i\in I(s)$.  

The graph $\bE$ is connected because any two triangulations of $P$ are related by a finite chain of   flips. It is also locally finite, and indeed $n$-regular. We now give $\bE$ the structure of an exchange graph as in Definition \ref{defn}. Given two edges $i,j\in I(s)$, we define  $v_{i,j}=+1$ (respectively $v_{i,j}=-1$) if the arcs $A(i)$ and $A(j)$  occur in anti-clockwise (respectively clockwise) order in  a triangle of $T(s)$. If $A(i)$ and $A(j)$  do not occur in the same triangle we set $v_{i,j}=0$. The edge bijection $\rho_i\colon I(s_1)\to I(s_2)$ associated to an edge $i\colon s_1\to s_2\in I$ is induced by the unique  bijection  between the arcs of the triangulations $T(s_1)$ and $T(s_2)$ extending the identification $T(s_1)\,\setminus\, A(i)= T(s_2)\,\setminus\, A(\epsilon(i))$.
 
Given a vertex $p\in P$ of the polygon, there is a corresponding vertex  $s(p)\in V$  given by the unique triangulation of $P$ all of whose edges contain $p$.  The associated  quiver $Q=Q(s(p))$ is a coherent orientation of the A$_n$ Dynkin diagram. 
Conversely, given an arbitrary orientation $Q$ of the A$_n$ Dynkin diagram, we can glue a chain of triangles along their edges  to obtain a  triangulation of $P$ whose associated quiver is $Q$.

To define the field isomorphisms we consider the variety $M=(\bP^1)^{m}$ and the quotient $\cM=M/\PGL_2$ by the diagonal action. We view a point of $M$ as a map of sets $p\colon \bZ_m\to M$ which then enables us to label the vertices of $P$  with elements $p(i)\in \bP^1$. We define an open subset $M_0\subset M$ by insisting that $p(i)\neq p(i+1)$ for all $i\in \bZ_m$, and that the image of $p$ contains at least three points of $\bP^1$. The latter condition implies that $\cM_0=M_0/\PGL_2$ is a manifold.

Suppose we are given a vertex $s\in V$ and an edge $i\in I(s)$.  Joining the two triangles in $T(s)$ containing the arc $A(i)$ defines a quadrilateral whose vertices are a cyclically-ordered subset of the vertices of $P$.  Following  \cite[Section 1.3]{FG1} we define
$x(s)(i)$ to be the cross-ratio of the corresponding points $p(j)\in \bP^1$. Note that for the description of $\cX_{\bR_+}$ below it is important  to normalise the cross-ratio so that $\operatorname{CR}(\infty,-1,0,x)=x$. 

For each vertex $s\in V$ we can combine the cross-ratios $x(s)(i)$ associated to edges $i\in I(s)$ to define a rational  map \begin{equation}
    \label{birat}
\phi_{\cM}(s)\colon \cM_0\dashrightarrow (\bC^*)^{I(s)}.\end{equation}

It is not hard to show (see the nexct subsection) that this map is in fact birational.
 The associated transition functions $\phi_{\bC}(s_1,s_2):=\phi_{\cM}(s_2)\circ \phi_{\cM}(s_1)^{-1}$ then induce the required field isomorphisms $\phi_{\bC}(s_1,s_2)^*$ by pull-back.

\begin{thm}
\label{thm-A-exchange}
    When equipped with the above data, the unoriented graph $\bE$ is an exchange graph in the sense of Section \ref{exch}. The associated quiver mutation class contains all orientations of the A$_n$ quiver. The mapping class group is $\bZ_m$ and acts on $\bE$ by rotations of the polygon $P$.
\end{thm}

\begin{proof}
The mutation rules were verified in \cite[Section 10]{FG0}, see also \cite[Prop. 3.4]{All}. To verify the  minimality property (E5) note that a triangulation of $P$ is uniquely determined by  the sets of 4-tuples of vertices of $P$ that are the vertices of the quadrilaterals obtained by gluing triangles along a common edge. For any two different size $4$ subsets of the vertices of $P$, it is straightforward to find some point $p\in \cM_0$ such that the corresponding cross ratios are different. The hypothesis in the minimality condition (E5) implies that $a,b\in V$ correspond to triangulations which have the same set of cross ratios. This then implies that they have the same set of quadrilaterals, and hence that they are equal.

The rotation action of $\bZ_m$ on the polygon $P$ induces an action on the graph $\bE$ which rotates triangulations in the obvious way. This clearly preserves all the associated data and hence defines an injective  homomorphism $\bZ_m\to \bG$. To prove this is surjective take  $g\in \bG$ and consider a vertex $s\in V$ corresponding to a triangulation $T(s)$ consisting of all arcs incident with a fixed vertex $p\in P$. The associated quiver $Q(s)$ is the coherent orientation of the $A_n$ quiver. By definition of the group $\bG$  the quiver $Q(g(s))$ is isomorphic to $Q(s)$. But the only triangulations of $P$ whose associated quivers are isomorphic to $Q(s)$ are the rotations of $T(s)$. Thus  composing $g$ with a rotation, we can assume that $g$ fixes the vertex $s$. Since the quiver $Q(s)$ has no non-trivial automorphisms $g$ must also fix all the edges $i\in I(s)$. The fact that $g$ preserves the edge bijections then implies that $g$ is the identity.  \end{proof}

\subsection{Complex, real and tropical cluster spaces}
\label{122}

The following result appears explicitly in \cite{All}. For the convenience of the reader we include a sketch of the proof.

\begin{thm} 
 \label{thm-A-complex}
 There is an isomorphism of complex manifolds  $\cX_{\bC}\to \cM_0$ under which the  cluster charts $\phi_{\bC}(s)$ correspond to the maps $\phi_{\cM}(s)$. 
 \end{thm}

 \begin{proof}
We must show that for each vertex $s\in V$, the the  inverse of the map  \eqref{birat} defines an open embeddding \[\phi_\cM(s)^{-1}\colon (\bC^*)^{I(s)}\into \cM_0.\]
We must also show that the images of these embeddings cover the space $\cM_0$.

For the first claim take an element $x\in (\bC^*)^{I(s)}$. Choose an arc $i\in I(s)$ and consider the associated quadilateral $Q(i)$ in the triangulation  $T(s)$. The component $x(i)\in \bC^*$  defines four points in $\bP^1$ with at most two the same, unique up to $\PGL_2$, and we use these points to label the vertices of $Q(i)$. We then inductively construct the points $p(j)\in \bP^1$ by passing through the tree that is dual to the triangulation $T(s)$. The resulting point in $M_0$ is independent of the choices up to the action of $\PGL_2$. This defines a regular inverse to $\phi_{\cM}(s)$ which is therefore an open embedding.
For the second claim, note that  given a point $p\colon \bZ_m\to \bP^1$ of the space $M_0$, we can construct a triangulation $T(s)$ of $P$ so that $p$ lies in the image of $\phi_\cM(s)^{-1}$ inductively by starting with a triangle whose vertices correspond to  three distinct points $p(j)\in \bP^1$.
\end{proof}

Let $\Poly$ be the space  parameterising ideal hyperbolic polygons with $m$ distinct vertices up to isometry, together with a chosen vertex. This space embeds in $\cM_0$ by taking $p(j)\in \bR\cup\{\infty\}$ to be the position of the $j$th vertex counting round from the chosen one. Since the points $p(j)$ occur in increasing order, for any vertex $s\in V$ and edge $i\in I(s)$, the corresponding cross-ratio $x(s)(i)$ is real and positive. We then easily obtain
 
\begin{thm}
\label{poly}
The isomorphism of Theorem \ref{thm-A-complex} restricts to an isomorphism  of smooth manifolds  $\cX_{\bR_+}\to \Poly$. 
\end{thm}

Let $\Diag$ to be the space of weighted diagonals in the polygon $P$ as in \cite[Section 5.1]{GM}.  The elements are triangulations equipped with non-negative real numbers associated to each arc. It is helpful to think in terms of the dual trees, with the weights giving the lengths of the edges of the tree. Then the flipping operation happens naturally when an edge becomes length zero. In this way $\Diag$ can be viewed as the parameter space for trivalent metric trees with $n$ vertices. Fock and Goncharov \cite[Theorem 12.1]{FG0} proved the following result\todo{It would be nice to explain what the charts $\phi_{\bR^t}$ are geometrically. I  think what was written before (commented out) was a bit too vague to extract this information.}

\begin{thm}
\label{diag}
    There is an isomorphism of PL manifolds  $\cX_{\bR^t}\to \Diag$. 
\end{thm}


\subsection{Cluster stability space}
\label{123}

Let $\Pol$ denote the space of complex polynomials of the form
\[q(x)=x^{n+1}+a_{n-1}x^{n-1}+\cdots +a_1 x + a_0.\]
There is an action of $\bZ_{m}$ on $\Pol$ generated by the transformation $a_k\mapsto \omega^{k+2}\cdot a_k$. This preserves the open subset $\Pol_0\subset \Pol$ consisting of polynomials without repeated zeroes.

Let $\Quad$ denote the space of quadratic differentials on $\bP^1$ with a single pole of order $m+2$ and simple zeroes, considered up to the action of $\PGL_2(\bC)$. Up to the action of this group any such quadratic differential can be written in the form $q(x)\,dx^{\tensor 2}$ for some polynomial $q\in\Pol$. Two quadratic differentials of this form are in the same $\PGL_2(\bC)$ orbit precisely if the corresponding polynomials are in the same $\bZ_m$ orbit.

We let $\Quad^{\fr}$ denote the $\bZ_m$ cover of the space $\Quad$  given by a choice of  framing, namely a choice of one of the $m$ horizontal directions at the pole.  The space $\Quad^{\fr}$ can then be identified with $\Pol$ by taking the horizontal direction to be $\bR_+$. Under this identification the action of $\bZ_m$ on $\Pol$ corresponds to the obvious action of $\bZ_m$ on $\Quad^{\fr}$ rotating the framing. 

Let us also introduce the open subset $\Quad_0^{\fr}\subset \Quad^{\fr}$ consisting of framed quadratic differentials with simple zeroes.
The following result can be obtained  by combining Theorem \ref{stab123} with the general results of \cite{BS}. For the particular examples considered here  see  \cite[Section 12.1]{BS} and also \cite{All}. 

\begin{thm}
\label{above}
    There is an isomorphism of complex manifolds  $\cS\to \Quad^{\fr}_0$.  The flat co-ordinates for the integral affine structure on $\cS$ correspond to integrals $ z_i=\int_{\gamma_i} \sqrt{q(x)} \, dx$ along paths connecting zeroes of $q(x)$.
\end{thm}

The isomoorphism of Theorem \ref{above}  intertwines the action of the cluster modular group $\bG=\bZ_m$ on $\cS$  with  the  action of $\bZ_m$ on $\Quad_0^{\fr}$ discussed above.
 
\subsection{Log cluster space}
\label{124}

In this section we assume that the reader is familiar with the notion of a projective structure on a Riemann surface, and the extension of this to the meromorphic setting \cite{AB}. 
Let $\Proj$ denote the space of meromorphic projective structures on $\bP^1$ with a single pole of order $m+2$. We also consider the $\bZ_m$ cover $\Proj^{\fr}$ consisting of projective structures equiiped with  a framing, namely a choice of one of the $m$  horizontal direction at the pole. Given a polynomial $q\in \Poly$, we get a meromorphic projective structure on $\bP^1$ by considering a ratio of two linearly independent solutions to the differential equation
\begin{equation}
    \label{schrod}
y''(x)=q(x)y(x).\end{equation}
The positive real axis gives a canonical choice of horizontal direction for this projective structure. This construction gives an identification between   $\Pol$ and the space $\Proj^{\fr}$.  Once again, the action of the group $\bZ_{m}$ on $\Pol$ corresponds to the obvious $\bZ_m$ action on $\Proj^{\fr}$ rotating the framing. 

 The space of solutions of the equation \eqref{schrod} is a two-dimensional vector space. Consider the decomposition of $\bC$ into closed sectors of angle $2\pi/5$ centered on the rays through the vertices of the polygon $P$. Analytic results \cite[Chapter 2]{Sib} show that for each such sector  there is a one-dimensional space of solutions $y(x)$ such that $y(x)\to 0$ as $x\to \infty$ in this sector. Such solutions are called subdominant.  There is then a map
\begin{equation}
    \label{crowntip}
m\colon \Proj^{\fr}\to \cM_0\end{equation}
which sends the framed projective structure defined by the equation \eqref{schrod} to the collection of lines in the space of solutions defined by the subdominant solutions in the various sectors.  This is referred to as the crown-tip map in \cite{GM}. It is known to be a surjective local homeomorphism: see \cite[Theorem 7.2]{B} and \cite[Theorem 39.1]{Sib}. 

Gupta and Mj \cite{GM} construct a homeomorphism
\begin{equation}
    \label{graft}
\Poly\times\Diag\to \Proj^{\fr}\end{equation}
which they call the {grafting map}.\todo{Ideally we would say something more about this.} The following result is an immediate consequence of their construction and the definition of the log cluster space.

\begin{thm}
\label{frlog}
There is an isomorphism of complex manifolds $\cL\to \Proj^{\fr}$.  Under this identification the map $ \expl\colon \cL\to \cX_{\bC}$ corresponds to the crown-tip  map \eqref{crowntip}.
\end{thm}

\begin{proof}
    Putting together Theorems \ref{diag} and \ref{poly}, the grafting map becomes a homeomorphism
    \begin{equation}
    \label{graft2}\cL=\cX_{\bR_+}\times \cX_{\bR^t}\to \Proj^{\fr}.\end{equation}
     What remains is to show is that, under this identification, the exponential map $ \expl\colon \cL\to \cX_{\bC}$ corresponds to the crown-tip map \eqref{crowntip}. Note that this ensures that \eqref{graft2} is a biholomorphism.
    
    A point of $\Diag$ is a triangulation $T(s)$ with a collection  of  weights $\theta(j)\in \bR_{\geq 0}$ associated to the edges $j\in I(s)$. A point of $\Poly$  has cross-ratio co-ordinates $r(j)\in \bR_+$ for edges $j\in I(s)$.  Now perform the grafting operation \eqref{graft}, apply the crown-tip map \eqref{crowntip} and compute the resulting cross-ratios $x(j)\in \bC^*$ for edges $j\in I(s)$. The claim is that $x(j)=r(j)\cdot e^{i\theta(j)}$. Indeed, this is verified in the section `Grafting an ideal quadrilateral' on page 33-34 of \cite{GM}.\todo{Strictly speaking there seems to be a sign in front of the $\theta$?}
\end{proof}

Under the identification of Theorem \ref{frlog} the action of the cluster modular group again corresponds to the  action of $\bZ_{m}$ rotating the framing.


\begin{appendix}

\section{Deformations of fans}
\label{fan-def-section}
In this Appendix we prove a result Proposition \ref{fan-prop} about deformations of complete fans in a fixed real vector space. Roughly speaking, it states that if we  start with a complete fan and then move the constituent cones in such a way that nothing goes wrong in codimension one, then the fan remains complete. More precisely, we require that if a pair of maximal cones meet in a common facet in the original fan, then this remains the case throughout the deformation. This result was applied in Section \ref{tangent} to prove Proposition \ref{sti}. 

\subsection{Families of cones and fans}
We fix a real vector space  $N_{\bR}\isom \bR^d$ throughout. By a cone $\sigma\subset N_\bR$ we mean a strongly convex polyhedral cone. 
We begin by considering families of cones $\sigma(t)$ in  $N_{\bR}$ depending on a parameter $t\in [0,1]$. Even if the cone $\sigma(t)$ is generated by vectors $n(t)$ which vary continuously with $t$, the number of faces of $\sigma(t)$ and even its dimension could fail to be constant. We therefore introduce the following more constrained notion.

\begin{defn}A combinatorially constant (c.c.) family of cones  $\sigma(t)$ in  $N_{\bR}$ depending on a parameter $t\in [0,1]$ is a family of the form  $\sigma(t)=\Phi(t)(\sigma)$, where $\sigma\subset N_{\bR}$ is a fixed cone and $\Phi(t)\in \Aut(N_{\bR})$ is a continuously varying family of linear automorphisms.\end{defn}  

The following result ensures that there is a well-defined notion of a face of a c.c.~family of cones $\sigma(t)$, and that such faces are in bijection with the  faces  of the cone $\sigma(t_0)$ for any fixed $t_0\in [0,1]$. 

\begin{lemma}
\label{le}
    Let $\sigma(t)$ be a c.c.~family of cones and let $\tau\subset \sigma(t_0)$ be a face. Then there is a unique c.c.~family of cones $\tau(t)$ such that $\tau(t_0)=\tau$ and $\tau(t)$ is a face of $\sigma(t)$ for all $t\in [0,1]$. 
\end{lemma}

A \emph{fan}   is a set of cones   that  is closed under taking faces and intersections. We say that a collection of cones $\sigma_1,\ldots,\sigma_k$ \emph{generates a fan} if they are all distinct and of dimension $d$, and the union  of the sets of faces of $\sigma_1,\ldots,\sigma_k$ is a fan.
We define a \emph{c.c.~family of fans} to be a set of c.c.~families of cones  which forms a fan $\Sigma(t)$ for  every $t\in [0,1]$. 
We say that a set $\sigma_1(t),\dots,\sigma_k(t)$ of c.c.~families of cones \emph{generates} a c.c.~family of fans if for every $t_0\in [0,1]$ the set of cones $\sigma_1(t_0),\dots,\sigma_k(t_0)$ generates a fan.

We defer the proof of the following result to the next subsection.

\begin{proposition} 
\label{fan-prop}
Let $\sigma_1(t),\dots, \sigma_k(t)$ be a set of c.c.~familes of cones such that for some  $t_0\in [0,1]$ the following hold:
\begin{itemize}
\item[(i)]  the cones $\sigma_i(t_0)$ generate   a complete fan,
\item[(ii)]  if  for some pair of indices $i,j$ the intersection $\sigma^{i}(t_0)\cap\sigma_j(t_0)$ is a facet of both the cones $\sigma^{i}(t_0)$ and $\sigma^{j}(t_0)$, then the same is true for all  $t\in [0,1]$. 
\end{itemize}
Then the c.c.~families of cones $\sigma_i(t)$ generate a c.c.~fan deformation,  and moreover the resulting fans $\Sigma(t)$ are complete for every $t\in [0,1]$.
\end{proposition}

The following example demonstrates the necessity of assumptions (i) and (ii).

\begin{example}
 For $t\in [0,1]$ and  $0\leq k \leq 3$ we set  $n_k(t)=e^{2\pi  i kt}$ and let $\rho_k(t)$ be the ray in $\bC\cong\bR^2$ generated by $n_k(t)$.
Let $\sigma_k(t)$ be the cone 
generated by $\rho_k(t),\rho_{k-1}(t)$. 
We consider the set of cones \[\Sigma(t)=\left\{\{0\},\rho_0(t),\rho_1(t),\rho_2(t),\rho_3(t), \sigma_1(t),\sigma_2(t),\sigma_3(t)\right\},\]  and extend the notions of families of cones and fans to include families parameterised by an interval $I\subset [0,1]$ in the obvious way. 

Restricting to the interval $(0,1/3]$ gives a c.c.~family of fans   which is complete for $t=1/3$ but not for $t<1/3$. Note that, since $\rho_0(t)=\rho_3(t)$ at $t=1/3$, the fan at $t=1/3$ is combinatorially different from the non-complete fans for $t<1/3$. This does not contradict Proposition \ref{fan-prop} applied to the set of families of cones $\{\sigma_1(t),\sigma_2(t),\sigma_3(t)\}$ with $t_0=1/3$ because  assumption (ii) fails. Indeed, the cones  
$\sigma_1(t),\sigma_3(t)$ share the facet $\rho_0(t)=\rho_3(t)$ for $t_0=1/3$ but not for any other value of $t$. 
Assumption (ii) \emph{is} satisfied for every other choice of $t_0\in (0,1/3]$ but then the completeness assumption (i) fails.

For $1/3<t<1$, the collection of c.c.~families of cones $\{\sigma_1(t),\sigma_2(t),\sigma_3(t)\}$ \emph{does not} generate a family of fans since $\sigma_1(t)\cap\sigma_3(t)$ has non-empty interior.
Finally and worse, upon including $t=1$, the family of cones $\sigma_k(t)$ is not c.c.
\end{example}

\subsection{Proof of Proposition~\ref{fan-prop}}

Assumptions (i) and (ii) together imply that every codimension one face $\tau({t_0})$ in the complete fan $\Sigma({t_0})$ generated by $\sigma_1(t_0),\dots, \sigma_k(t_0)$  extends uniquely to  a c.c.~family of  cones $\tau(t)$ so that if $\tau({t_0})\subset \sigma({t_0})$ is a face with $\sigma({t_0})$ $d$-dimensional then $\tau({t})\subset \sigma({t})$ is a face for all $t\in [0,1]$.
This statement can be generalised as follows.

\begin{lemma}
\label{greg}
     Every cone $\tau({t_0})\in \Sigma({t_0})$ extends uniquely to  a c.c.~family of cones $\tau(t)$ in such a way that  if $\tau({t_0})\subseteq \tau'({t_0})$ is a face then 
$\tau(t)\subseteq \tau'(t)$ is a face for all $t\in [0,1]$.
\end{lemma}

\begin{proof}
We use descending induction on dimension with the base case being given by assumption (ii).
Assume the claim has been proven for all cones of dimension $k$ and higher. Let $\tau(t_0)\in \Sigma(t_0)$ be a $(k-1)$-dimensional cone. Let $\Sigma^{\geq k}_{\tau(t_0)}$ denote the set of cones of dimension $\geq k$ in $\Sigma({t_0})$ which contain $\tau({t_0})$. The c.c.~family of cones $\tau'(t)$ associated to every $\tau'({t_0})\in \Sigma^{\geq k}_{\tau({t_0})}$ contains a c.c.~family of cones deforming  $\tau({t_0})$ and we need to show that the a priori different deformations of $\tau({t_0})$ actually agree. 
Let  $\tau'(t),\tau''(t)$ be the c.c.~family of cones deforming $\tau'({t_0}),\tau''({t_0})\in \Sigma^{\geq k}_{\tau({t_0})}$, and let $\tau'_\tau(t),\tau''_\tau(t)$ be the induced c.c.~cone deformations of $\tau$ inside them. We can  find a sequence of $d$-dimensional cones $\sigma_1(t_0),...,\sigma^r(t_0)$ in $\Sigma^{\geq k}_{\tau(t_0)}$ with consecutive cone deformations sharing a facet,  and $\sigma_1(t_0)$ containing $\tau'(t_0)$ and $\sigma^r(t_0)$ containing $\tau''(t_0)$.
Since every $\sigma^{j}(t_0)$ contains $\tau(t_0)$ and consequently also $\sigma^{j}(t_0)\cap \sigma^{j+1}(t_0)$ contains $\tau(t_0)$, the  deformations of  $\tau(t_0)$ induced by all  cones $\sigma_j(t)$ in the sequence coincides. Hence they also coincide for $\tau'(t)$ and $\tau''(t)$ which gives the result.
\end{proof}

Lemma \ref{greg} gives a set $\Sigma$ of c.c.~families of cones whose specialization to $t=t_0$ yields the fan $\Sigma(t_0)$. To prove Proposition~\ref{fan-prop} it remains to show that 
$\Sigma$ is a c.c.~family of fans. If we glue all the cones in $\Sigma(t)$ along their faces as prescribed by $\Sigma(t)$, we get the topological realization $|\Sigma(t)|$ of the polyhedral complex $\Sigma(t)$. We do not know at this point whether the natural map $\pi(t)\colon|\Sigma(t)|\to N_\bR$ is bijective. We do know that its restriction to each cone is injective. We also know that $\pi(t)$ is a homeomorphism for $t=t_0$ by assumption (i).
We next make a topological consideration.

Removing the origin in $N_\bR$ and dividing by the positive scaling action gives a projection $N_\bR\cong\bR^d\supset \bR^d\,\setminus\,\{0\} \to S^{d-1}$. We similarly obtain a quotient 
$Z(t) = \left(|\Sigma(t)|\,\setminus\,\{0\}\right)/\bR_+$ as well as a map $Z(t) \to S^{d-1}$.
After choosing an orientation of $N_\bR$, the map $Z(t)\to S^{d-1}$ is an oriented map of topological manifolds which is a homeomorphism for $t=t_0$. Since the degree is a homotopy invariant, $Z(t)$ is a fundamental cycle of $S^{d-1}$ for every $t$.

If $\omega$ is an orientation form on $S^{d-1}$ then $\int_{S^{d-1}} \omega$ agrees with the integral of $\omega$ over any fundamental cycle. 
The integral of $\omega$ over the projection of every maximal cone in $\Sigma(t)$ to $S^{d-1}$ is positive for all $t$. This observation implies that the interiors of two distinct cones $\sigma_i(t)$ must have empty intersection in $N_\bR$ for every $t$.

It remains to show that $\pi(t)\colon|\Sigma(t)|\to N_\bR$ is injective for all $t$. 
We do so by descending induction on the cone dimension. 
We just proved that the restriction of $\pi(t)$ to the union of the interiors of all $d$-dimensional cones is injective. 
Assume that we proved already that the restriction of $\pi(t)$ to the union $U_{k+1}$ of relative interiors of all cones of dimension $\geq k+1$ is injective. 
Let $\tau\in \Sigma(t)$ be a cone of dimension $k$ and let $U_\tau$ be the open star in $|\Sigma(t)|$ of the relative interior of $\tau$. We claim that the restriction of $\pi(t)$ to $U_\tau$ is injective. Indeed by the induction hypothesis, $\pi(t)$ is injective on $U_\tau\,\setminus\, \tau=U_\tau\cap U_{k+1}$. We also know that $\pi(t)$ is injective on $\tau$, so the only possibility for injectivity to fail is the situation where a point $x$ in the relative interior of $\tau$ and a point $x'\in U_\tau\,\setminus\, \tau$ map to the same point in $N_\bR$. Let $\sigma\in\Sigma(t)$ be a maximal cell containing $x'$ then $\sigma$ also contains $\tau$ and thus $x$. Since $\pi(t)$ is injective on $\sigma$, the points $x,x'$ cannot have the same image. We just showed that every open star of a face in $U_k$ maps injectively into $N_\bR$. In particular, every point in $U_k$ has an open neighbourhood that maps homeomorphically to $N_\bR$. Since $\pi(t)$ has degree 1 and $\pi(t)$ is orientation-preserving, we conclude that the restriction of $\pi(t)$ to $U_k$ is injective.\qed
\end{appendix}


\bibliographystyle{amsplain}

\end{document}